\documentclass[12pt,a4paper]{article}

\usepackage{amssymb}
\usepackage{amsfonts}
\usepackage{theorem}
\usepackage{bbm}
\usepackage{graphicx}
\usepackage{amsmath}
\usepackage{enumerate}
\usepackage{indentfirst}
\usepackage{color}
\usepackage{hyperref}

\newtheorem{theorem}{Theorem}
\newtheorem{remark}{Remark}
\newtheorem{lemma}{Lemma}
\newtheorem{prop}{Proposition}

\newenvironment{proof}[1][Proof]{\textbf{#1.} }{\hfill $\square$}

\newcommand{\R}{\mathbbm{R}}

\newcommand{\eps}{\varepsilon}

2

\title{Homogenization of nonlocal equations in randomly evolving media. Diffusion approximation.}
\author{
M. Kleptsyna \thanks{Le Mans Universit\'e, Laboratoire Manceau de Math\'ematiques, Avenue Olivier Messiaen, 72085 Le Mans, Cedex 9, France.\hfill\break
e-mail: {\tt marina.kleptsyna@univ-lemans.fr}\hfill\break
e-mail: {\tt alexandre.popier@univ-lemans.fr}
}
\,\,\, A. Piatnitski \thanks{The Arctic University of Norway, campus Narvik, Norway \\  and \ Higher School of Modern Mathematics MIPT, Moscow, Russia\hfill\break
e-mail: \ {\tt apiatnitski@gmail.com}
\hfill\break
}
\,\,\, and A. Popier\footnotemark[3]
}

\date{\today}

\begin{document}
\maketitle

\begin{abstract}
The paper deals with homogenization and higher order approximations of solutions to  nonlocal evolution equations
of convolution type whose coefficients are periodic in the spatial variables and random stationary in time.  We assume that the convolution kernel  has finite moments up to order three. Under proper mixing assumptions, we study the limit behavior of the normalized difference between solutions of the original and the homogenized problems and show that this difference converges to the solution of a linear stochastic partial differential equation.
\end{abstract}




\section{Introduction}

In the paper we study homogenization problem for evolution equations of the form
\begin{equation}\label{intr_eq}
\partial_ t u(x,t)=
 \eps^{-(d+2)}\int_{\R^d} a\left(\frac{x-y}\eps \right) \Lambda\left( \frac x\eps, \frac y\eps,\frac t{\eps^\alpha} \right) \left( u(y,t)-u(x,t)\right) dy,
\end{equation}
$(x,t)\in\mathbb R^d\times(0,T]$.
Here $\eps$ is a small positive parameter, $a$ is a non-negative integrable function in $\mathbb R^d$ that has finite
moments up to order three, and $\Lambda(\xi,\eta,s)$ is a positive bounded function which is periodic in $\xi$ and $\eta$
and random stationary ergodic in $s$.

Convolution type operators appear in various applications in such fields as population biology and mechanics of porous media.
One of the models widely used in population biology is the so-called contact model. In this model Equation
\eqref{intr_eq} describes the evolution of the density of a population, see \cite{KKP2008}.

In the mathematical literature homogenization problems for zero-order convolution type operators with periodic
coefficients were studied in \cite{PiaZhi17} and, in the non-symmetric case, in \cite{piat_zhiz_19aa}. It was shown that
in the symmetric case the problem admits homogenization, and in the limit one has a second order elliptic differential operator with
constant coefficients.  In the non-symmetric case the homogenization result holds in moving coordinates.
$\Gamma$-convergence problems for convolution-type functionals in media with a periodic microstructure have been addressed in \cite{BraPia22} and \cite{BCPE2021}.
The work \cite{MR4053030} deals with stochastic homogenization of symmetric convolution-type operators. It is proved that
in statistically homogeneous ergodic media the homogenization result holds almost surely, and the limit operator is a deterministic second order elliptic operator with constant coefficients.

\medskip

A Cauchy problem for Equation \eqref{intr_eq} with $\Lambda(\xi,\eta,s)$ being periodic in $\xi,\ \eta$ and stationary
ergodic in $s$ was investigated in the recent work \cite{piat:zhiz:25}. If $\Lambda(\xi,\eta,s)=\Lambda(\eta,\xi,s)$ for
all $s$, then the problem admits homogenization almost surely, and the effective equation is deterministic.
In the non-symmetric case, under additional mixing conditions, the  solutions of the $\eps$-problem converge in law
in a properly chosen moving frame, the limit equation being a stochastic partial differential equation with a finite-dimensional
multiplicative noise.
If the limit is deterministic, it is natural to raise the questions about the rate of convergence and the behavior of the corresponding normalized difference. In the present work we address this question. To our best knowledge diffusion approximation problems for convolution type operators with oscillating coefficients
have not been studied in the existing literature.

\medskip

In the case of parabolic differential operators whose coefficients are rapidly oscillating functions of both  spacial and temporal variables, this question has already been thoroughly investigated.
To begin with, homogenization problems for divergence form parabolic differential operators with the coefficients which are periodic both in spacial and temporal variables were considered in \cite{MR503330} and \cite{zbMATH03921085}. 
The work \cite{DP05} focused on operators with lower order terms under diffusion scaling of the coefficients, while \cite{Garn} dealt with homogenization of convection-diffusion equations under non-diffusion scaling.

Parabolic differential equations with coefficients which are periodic
in spatial variables and
random stationary in time were studied in \cite{KP_2002} in the case of diffusive scaling $\alpha=2$ and in \cite{KP_1995} for $\alpha\not=2$. 
In the case of a deterministic limit problem the behavior of the normalized difference between solutions of the original
and homogenized problems was investigated in \cite{KPP_2015}: 
under additional mixing assumptions the studied normalized difference converges in law
to a solution of the limit stochastic partial differential equation (SPDE in short). This result has been extended for the non-diffusive scaling $\alpha \neq 2$ in the papers \cite{KPP22} and  \cite{KPP24}.

\medskip
In the present paper, we only consider the diffusive scaling $\alpha=2$, the other cases are left for further research. A detailed description of the framework of our study is provided in Section \ref{s_setup}.

First we study the symmetric case $a(-z)=a(z)$, $z\in\mathbb R^d$,
and $\Lambda(\xi,\eta,s)=\Lambda(\eta,\xi,s)$ for all $\xi, \, \eta, \, s$. As was proved in \cite{piat:zhiz:23a},
in this case a solution to the Cauchy problem for  equation \eqref{intr_eq} equipped with the initial condition $u(\cdot,0)=\imath$,
$u_0\in L^2(\mathbb R^d)$, converges in $L^2(\mathbb R^d\times[0,T])$ to a solution  of the following Cauchy problem
$$
\partial_t u=\mathrm{div}\big(\Theta^{\text{eff}}\nabla u\big), \quad u(\cdot,0)=\imath;
$$
 where $\Theta^{\text{eff}}$ is a positive definite constant deterministic matrix. Our first goal in this paper is to study the asymptotic behaviour of the normalized difference
 \begin{equation}\label{no_di}
U^\eps:= \frac1\eps \big(u^\eps-\hat u^\eps\big),
 \end{equation}
 where $u^\eps$ is a solution of the original Cauchy problem, and $\hat u^\eps$ is the sum of the two leading terms of the asymptotic expansion of $u^\eps$.  We show that under additional assumptions on the rate of decay of the mixing coefficients of $\Lambda$ the function $U^\eps$ converges in law as $\eps\to0$ to a solution of the limit stochastic partial differential equation with an additive noise (Theorem \ref{thm:main_result}). Further details can be found in Section \ref{s_sym_case}.

 In the non-symmetric case the homogenization takes place in moving coordinates, and the limit problem need not
 be deterministic, see \cite{piat:zhiz:25}. In this case we consider only those kernels for which the limit problem is deterministic.
 For such kernels we define the normalized difference $U^\eps$ in the same way as in \eqref{no_di}, and study
 the asymptotic behaviour of $U^\eps$ as $\eps\to0$. Also, we provide sufficient conditions under which the limit
 is deterministic. This is the subject of Section \ref{s_no_sy} and the main result is Theorem \ref{thm:main_result_non_symm}.

%


\section{The setup} \label{s_setup}

Let $(\Omega,\mathcal{F},{\bf P})$ be a standard probability space equipped with a measure preserving ergodic dynamical system $\mathcal{T}_s$, $s\in\mathbb R$.

Given a measurable function $\boldsymbol{\Lambda}(x,y,\omega)$, $(x,y)\in\mathbb R^d\times\mathbb R^d$, $\omega\in\Omega$, which is periodic in $(x,y)$ variables with a period one in each coordinate direction, we define
a random field $\Lambda(x,y,s)$ by
$$
\Lambda(x,y,s)=\boldsymbol{\Lambda}(x,y,\mathcal{T}_s\omega).
$$
Then $\Lambda(x,y,s)$ is periodic in $(x,y)$ and stationary ergodic  in $s$.

We consider the  operator
\begin{equation} \label{eq:operator}
L_t  \phi(x)  =  \int_{\R^d} a\left( x-y \right) \Lambda\left( x, y, t \right) \left( \phi(y)-\phi(x)\right) dy,
\end{equation}
the rescaled operator
\begin{equation*}
L^\eps_t  \phi(x)  = \dfrac{1}{\eps^{d+2}} \int_{\R^d} a\left( \dfrac{x-y}{\eps}\right) \Lambda\left( \dfrac{x}{\eps},  \dfrac{y}{\eps},  \dfrac{t}{\eps^2} \right) \left( \phi(y)-\phi(x)\right) dy,
\end{equation*}
and the following Cauchy problem in $\mathbb R^d\times(0,T]$, $T>0$\,:
\begin{equation} \label{eq:u^eps}
\frac{\partial u^\eps}{\partial t} = L^\eps_t u^\eps ,   \qquad
u^\eps (x,0) =  \imath(x)
\end{equation}
with a small positive parameter $\eps $.

We assume that the coefficients in (\ref{eq:u^eps}) possess the following properties.
\begin{itemize}
\item[{\bf H1}]
The function $a : \R^d \to \R$ is non-negative and
$$
\int_{\R^d} a(z) dz = 1,\quad \int_{\R^d} a\left( z \right) |z|^3 dz < +\infty.
$$
%
\item[{\bf H2}] $\Lambda$ is periodic both in $x$ and $y$ variables. Moreover, there are constants $\Lambda^->0$ and $\Lambda^+\geqslant \Lambda^-$ such that, for any $(x,y,s)$,
$$
0< \Lambda^-\leq \Lambda(x,y,s) \leq \Lambda^+.
$$
\end{itemize}
Without loss of generality, we assume that the period is equal to one for each coordinate direction.

Setting
$\mathcal{F}_{\leq r}=\sigma\{\Lambda(x,y,s)\,:\,s\leq r\}$ and
$\mathcal{F}_{\geq r}=\sigma\{\Lambda(x,y,s)\,:\,s\geq r\}$, we define
$$
\rho(r)=\sup\limits_{\xi_1,\xi_2}{\bf E}(\xi_1\xi_2)
$$
where the supremum is taken over all $\mathcal{F}_{\leq 0}$-measurable $\xi_1$ \, and $\mathcal{F}_{\geq r}$-measurable
$\xi_2$ such that ${\bf E}\xi_1={\bf E}\xi_2=0$, and  ${\bf E}\{(\xi_1)^2\}={\bf E}\{(\xi_2)^2\}=1$.  We then assume that
\begin{itemize}
\item[{\bf H3}] The function $\rho$ satisfies the estimate
$
\displaystyle \int_0^\infty\rho(r)dr<+\infty.
$
\item[{\bf H4}] $\imath$ is a Schwartz class function. 
\end{itemize}

In this paper $\nabla^{(n)}$ stands for the iterating operator: $\underbrace{\nabla \cdots \nabla}_{n\ \mbox{{\tiny times}}}$.

\section{The symmetric case}
\label{s_sym_case}

In this section we assume that the coefficients $a(\cdot)$ and $\Lambda(\cdot)$ satisfy the following symmetry conditions:
\begin{itemize}
\item[{\bf H5}]  The function $a$ is even, and $\Lambda$  is symmetric: $\Lambda(x,y,s) = \Lambda(y,x,s)$.
\end{itemize}

We define the unique periodic in $\xi$ and stationary solution of the equation:
\begin{equation} \label{eq:first_corrector}
\partial_s \chi_1 (\xi, s) - (L_s \chi_1 ) \left( \xi,  s \right) =-\int_{\R^d} a\left( z\right) \Lambda\left( \xi,  \xi -z,  s \right) z dz,
\end{equation}
such that $\int_{\mathbb T^d} \chi_1(\xi,s)d\xi = 0$.
The main result of this section is
\begin{theorem} \label{thm:main_result}
Under conditions {\rm \bf{H1}}--{\rm \bf{H5}} , there exist
\begin{itemize}
\item a constant positive definite matrix $\Theta^{{\rm eff}}$ (given by \eqref{eq:effective_matrix}),
\item a symmetric constant positive semi-definite $d^2\times d^2$ matrix $A^{{\rm eff}} $
\item a constant tensor $H^{\rm eff}=\{H^{\rm eff}_{jkl}\}_{j,k,l=1}^d$ (given by \eqref{eq:effective_drift}),
\end{itemize}
such that the process $\dfrac{1}{\eps} (u^\eps-u^0) -  \chi_1 \left( \dfrac{x}{\eps}, \dfrac{t}{\eps^2} \right) \nabla u^0$ converges in law in $L^2(\mathbb R^d\times(0,T))$ to the unique solution of the following SPDE:
\begin{equation}\label{eq:limit_SPDE}
\begin{array}{c}
\displaystyle
dv_t - {\rm div} ( \Theta^{{\rm eff}} \nabla v) dt = A^{{\rm eff}} \nabla \nabla u^0\,dW_t + H^{{\rm eff}} \nabla \nabla \nabla u^0 dt,\\
v(x,0)=0,
\end{array}
\end{equation}
where $u^0$ satisfies the equation
\begin{equation} \label{eq:homo_pb}
\partial_t u^0 - {\rm div} ( \Theta^{{\rm eff}} \nabla u^0) = 0,
\end{equation}
with the initial condition $\imath$.
\end{theorem}

In the rest of this section we prove this result.  We consider the following formal asymptotic expansion for $u^\eps$:
\begin{align} \label{eq:expansion}
u^\eps(x,t) & = u^0(x,t) + \eps \chi_1 \left( \dfrac{x}{\eps}, \dfrac{t}{\eps^2} \right) \nabla u^0(x,t) \\ \nonumber
& + \eps^2 \chi_2\left( \dfrac{x}{\eps}, \dfrac{t}{\eps^2} \right) \nabla\nabla  u^0(x,t) + R^\eps(x,t).
\end{align}
Our goal is to introduce a stationary corrector $\chi_2$  in such a way  that $\dfrac{1}{\eps} R^\eps$ converges in a suitable space.  Let us first derive the equations satisfied by $\chi_2$ and $R^\eps$.

\subsection{Equation for the remainder $R^\eps$}

First note that for an arbitrary function $\kappa=\kappa\left( \dfrac{x}{\eps}, x\right)$, we have
\begin{align*}
L^\eps_t \kappa & =  \dfrac{1}{\eps^{d+2}} \int_{\R^d} a\left( \dfrac{x-y}{\eps}\right) \Lambda\left( \dfrac{x}{\eps},  \dfrac{y}{\eps},  \dfrac{t}{\eps^2} \right) \left( \kappa\left( \dfrac{y}{\eps}, y\right) - \kappa\left( \dfrac{x}{\eps}, x\right)\right) dy \\
& = \dfrac{1}{\eps^{2}} \int_{\R^d} a\left( z\right) \Lambda\left( \dfrac{x}{\eps},  \dfrac{x}{\eps} -z,  \dfrac{t}{\eps^2} \right) \left( \kappa\left( \dfrac{x}{\eps}-z, x-\eps z\right) - \kappa\left( \dfrac{x}{\eps}, x\right)\right) dz \\
& = \dfrac{1}{\eps^{2}} \int_{\R^d} a\left( z\right) \Lambda\left( \dfrac{x}{\eps},  \dfrac{x}{\eps} -z,  \dfrac{t}{\eps^2} \right) \left( \kappa\left( \dfrac{x}{\eps}-z, x\right) - \kappa\left( \dfrac{x}{\eps}, x\right)\right) dz \\
& +  \dfrac{1}{\eps^{2}} \int_{\R^d} a\left( z\right) \Lambda\left( \dfrac{x}{\eps},  \dfrac{x}{\eps} -z,  \dfrac{t}{\eps^2} \right) \left( \kappa\left( \dfrac{x}{\eps}-z, x-\eps z\right) - \kappa\left( \dfrac{x}{\eps}-z, x\right)\right) dz.
\end{align*}
Hence applying this for a function $\kappa$ of the form $\chi \left( \dfrac{x}{\eps}, \dfrac{t}{\eps^2} \right) \phi(x,t) $ leads to
\begin{align*}
L^\eps_t \kappa & =  \dfrac{1}{\eps^{2}} \int_{\R^d} a\left( z\right) \Lambda\left( \dfrac{x}{\eps},  \dfrac{x}{\eps} -z,  \dfrac{t}{\eps^2} \right) \left(\chi \left( \dfrac{x}{\eps} -z, \dfrac{t}{\eps^2} \right)  - \chi \left( \dfrac{x}{\eps}, \dfrac{t}{\eps^2} \right) \right) dz \phi(x,t)\\
&+  \dfrac{1}{\eps^{2}} \int_{\R^d} a\left( z\right) \Lambda\left( \dfrac{x}{\eps},  \dfrac{x}{\eps} -z,  \dfrac{t}{\eps^2} \right) \chi \left( \dfrac{x}{\eps} -z, \dfrac{t}{\eps^2} \right) \left( \phi\left( x- \eps z, t\right) - \phi\left(x, t\right)\right) dz \\
& =\dfrac{1}{\eps^2} (L_s \chi ) \left( \dfrac{x}{\eps},  s \right)\Big|_{s=\frac{t}{\eps^2}} \phi(x,t)  \\
& +  \dfrac{1}{\eps^{2}} \int_{\R^d} a\left( z\right) \Lambda\left( \dfrac{x}{\eps},  \dfrac{x}{\eps} -z,  \dfrac{t}{\eps^2} \right) \chi \left( \dfrac{x}{\eps} -z, \dfrac{t}{\eps^2} \right) \left( \phi\left( x- \eps z, t\right) - \phi\left(x, t\right)\right) dz .
\end{align*}
Notice that
$$L_t \phi(x) = L^1_t  \phi(x) =  \int_{\R^d} a\left( z \right) \Lambda\left( x, x-z, t \right) \left( \phi(x-z)-\phi(x)\right) dz.$$
If $\phi$ is a smooth function, then using Taylor's expansion, we obtain
\begin{align*}
\phi &\left( x- \eps z, t\right) - \phi\left(x, t\right) = - \eps \int_0^1 \nabla \phi(x-\eps q z,t) dq \\
& =- \eps z \nabla \phi(x,t) +\eps^2 \int_0^1 \nabla \nabla \phi(x - \eps q z, t) (z \otimes z) (1-q) dq\\
&=- \eps z \nabla \phi(x,t) +\frac12\eps^2 z\otimes z\nabla\nabla\phi(x)-\frac12\eps^3 \int_0^1 \nabla\nabla \nabla \phi(x - \eps q z, t) (z \otimes z\otimes z) (1-q)^2 dq.
\end{align*}
With the help of these relations we deduce that
\begin{align} \label{eq:space_deriv}
& L^\eps_t u^\eps(x,t) =- \dfrac{1}{\eps} \int_{\R^d} a\left( z\right) \Lambda\left( \dfrac{x}{\eps},  \dfrac{x}{\eps} -z,  \dfrac{t}{\eps^2} \right) z dz \nabla u^0 (x,t)\\ \nonumber
&\quad  +  \frac12\int_{\R^d} a\left( z\right) \Lambda\left( \dfrac{x}{\eps},  \dfrac{x}{\eps} -z,  \dfrac{t}{\eps^2} \right) (z \otimes z)  dz\, \nabla \nabla u^0(x , t) \\ \nonumber
&\quad  - \frac\eps2 \int_{\R^d} a\left( z\right) \Lambda\left( \dfrac{x}{\eps},  \dfrac{x}{\eps} -z,  \dfrac{t}{\eps^2} \right)  \int_0^1 \nabla \nabla u^0(x - \eps q z, t) (z \otimes z\otimes z) (1-q)^2 dq dz  \\ \nonumber
&\quad  + \dfrac{1}{\eps} (L_t \chi_1 ) \left( \dfrac{x}{\eps},  \dfrac{t}{\eps^2} \right) \nabla u^0(x,t) \\ \nonumber
&\quad  -  \int_{\R^d} a\left( z\right) \Lambda\left( \dfrac{x}{\eps},  \dfrac{x}{\eps} -z,  \dfrac{t}{\eps^2} \right) \chi_1 \left( \dfrac{x}{\eps} -z, \dfrac{t}{\eps^2} \right)\otimes z  \nabla \nabla u^0 \left( x, t\right)dz  \\ \nonumber
&\quad  -  \eps\int_{\R^d} a\left( z\right) \Lambda\left( \dfrac{x}{\eps},  \dfrac{x}{\eps} -z,  \dfrac{t}{\eps^2} \right) \chi_1 \left( \dfrac{x}{\eps} -z, \dfrac{t}{\eps^2} \right)\otimes z\otimes z \int_0^1 \nabla\nabla \nabla u^0 \left( x- \eps q z, t\right)(1-q) dq   dz \\ \nonumber
&\quad  +(L_t \chi_2 ) \left( \dfrac{x}{\eps},  \dfrac{t}{\eps^2} \right) \nabla\nabla u^0(x,t)  \\ \nonumber
&\quad  + \eps\int_{\R^d} a\left( z\right) \Lambda\left( \dfrac{x}{\eps},  \dfrac{x}{\eps} -z,  \dfrac{t}{\eps^2} \right) \chi_2 \left( \dfrac{x}{\eps} -z, \dfrac{t}{\eps^2} \right)\otimes z\int_0^1\nabla\nabla\nabla u^0 \left(x-\eps zq, t\right) \,dq dz \\ \nonumber
%
& \quad +L^\eps_t R^\eps(x,t) .
\end{align}
Moreover,
\begin{align} \label{eq:time_deriv}
\partial_t u^\eps(x,t) & = \partial_t  u^0(x,t) + \partial_t  R^\eps(x,t) + \dfrac{1}{\eps} \partial_s \chi_1 \left( \dfrac{x}{\eps}, \dfrac{t}{\eps^2} \right) \nabla u^0(x,t) \\  \nonumber
& + \partial_s \chi_2\left( \dfrac{x}{\eps}, \dfrac{t}{\eps^2} \right) \nabla\nabla  u^0(x,t) \\  \nonumber
& +  \eps \chi_1 \left( \dfrac{x}{\eps}, \dfrac{t}{\eps^2} \right) \partial_t \nabla u^0(x,t) + \eps^2 \chi_2\left( \dfrac{x}{\eps}, \dfrac{t}{\eps^2} \right) \partial_t \nabla\nabla  u^0(x,t) .
\end{align}
Due to the symmetry of $a$ and $\Lambda$, the function
\begin{equation}\label{def_g}
g(\xi, s)=-\int_{\R^d} a\left( z\right) \Lambda\left( \xi,  \xi -z,  s \right) z dz
\end{equation}
satisfies, for any $s\in\mathbb R$, the relation
$$
\int_{\mathbb T^d} g(\xi,s) d\xi = 0.
$$
Then, according to  \cite[Lemma 3.9 and Remark 3.3]{piat:zhiz:25}, 
Equation \eqref{eq:first_corrector} has a unique periodic in $\xi$ and stationary in $s$ solution $\chi_1\in L^\infty(-\infty,\infty;L^2(\mathbb T^d))$ .
Substituting in \eqref{eq:u^eps} the right-hand sides of  \eqref{eq:space_deriv} and \eqref{eq:time_deriv} for $L_t^\eps u^\eps$ and $\partial_t u^\eps$, respectively,     we deduce that if $\chi_1$ solves  \eqref{eq:first_corrector},
then the sum of the terms of order $\eps^{-1}$ in \eqref{eq:space_deriv} and \eqref{eq:time_deriv} vanishes, and we obtain
\begin{align} \label{eq:first_step}
& \partial_t  u^0(x,t) + \partial_t  R^\eps(x,t) + \partial_s \chi_2\left( \dfrac{x}{\eps}, \dfrac{t}{\eps^2} \right) \nabla\nabla  u^0(x,t) \\  \nonumber
& \qquad +  \eps \chi_1 \left( \dfrac{x}{\eps}, \dfrac{t}{\eps^2} \right) \partial_t \nabla u^0(x,t) + \eps^2 \chi_2\left( \dfrac{x}{\eps}, \dfrac{t}{\eps^2} \right) \partial_t \nabla\nabla  u^0(x,t) \\ \nonumber
& =  L^\eps_t R^\eps(x,t) +  \int_{\R^d} a\left( z\right) \Lambda\left( \dfrac{x}{\eps},  \dfrac{x}{\eps} -z,  \dfrac{t}{\eps^2} \right)(z \otimes z) dz\, \nabla \nabla u^0(x, t) \\ \nonumber
& \quad - \frac\eps2\int_{\R^d} a\left( z\right) \Lambda\left( \dfrac{x}{\eps},  \dfrac{x}{\eps} -z,  \dfrac{t}{\eps^2} \right)(z \otimes z\otimes z)
 \nabla\nabla\nabla u^0(x-\eps qz, t)(1-q)^2 dz \\ \nonumber
&\quad  -  \int_{\R^d} a\left( z\right) \Lambda\left( \dfrac{x}{\eps},  \dfrac{x}{\eps} -z,  \dfrac{t}{\eps^2} \right) \chi_1 \left( \dfrac{x}{\eps} -z, \dfrac{t}{\eps^2} \right)\otimes z dz\,\nabla\nabla u^0 \left( x, t\right)   \\ \nonumber
&\quad  - \eps \int_{\R^d} a\left( z\right) \Lambda\left( \dfrac{x}{\eps},  \dfrac{x}{\eps} -z,  \dfrac{t}{\eps^2} \right) \chi_1 \left( \dfrac{x}{\eps} -z, \dfrac{t}{\eps^2} \right)\otimes z\otimes z\int_0^1 \nabla\nabla\nabla u^0 \left( x- \eps q z, t\right)(1-q) dq   dz \\ \nonumber
&\quad  +(L_t \chi_2 ) \left( \dfrac{x}{\eps},  \dfrac{t}{\eps^2} \right) \nabla\nabla u^0(x,t)  \\ \nonumber
&\quad  +\eps \int_{\R^d} a\left( z\right) \Lambda\left( \dfrac{x}{\eps},  \dfrac{x}{\eps} -z,  \dfrac{t}{\eps^2} \right) \chi_2 \left( \dfrac{x}{\eps} -z, \dfrac{t}{\eps^2} \right)\otimes z\int_0^1 \left(  \nabla\nabla\nabla u^0\left( x- \eps q z, t\right)   \right) dz .
\end{align}
\begin{remark}
Here the third moment condition on $a(\cdot)$ was used. Without this condition we would only have the relation
$$
\int_{\R^d} a\left( z\right) \Lambda\left( \dfrac{x}{\eps},  \dfrac{x}{\eps} -z,  \dfrac{t}{\eps^2} \right)  \int_0^1 \left(  \nabla \nabla u^0(x - \eps q z, t)-\nabla\nabla u^0(x,t) \right) (1-q) dq (z \otimes z) dz = o(1),
$$
which is sufficient for homogenization of the studied problem, but not for the diffusion approximation.
\end{remark}
We define the following matrix-valued functions:
\begin{align} \nonumber
h(\xi,s) & =  \int_{\R^d} a\left( z\right) \Lambda\left( \xi,  \xi -z,  s \right)  \left( \dfrac{1}{2} (z \otimes z) - \chi_1 \left( \xi -z, s \right)\otimes z  \right)   dz, \\ \label{eq:effective_matrix}
\Theta(s) & = \int_{\mathbb T^d}  h(\xi,s) d\xi , \quad \Theta^{{\rm eff}} = \mathbb E (\Theta(s)) \quad \widetilde \Theta(s) = \Theta(s) - \Theta^{{\rm eff}} .
\end{align}
By stationarity, $\Theta^{{\rm eff}}$ does not depend on $s$.
By \cite[Lemma 3.8]{piat:zhiz:25}, there is a unique solution $\chi_2$ to the equation:
\begin{equation} \label{eq:second_corrector}
\partial_s \chi_2 (\xi, s) - (L_t \chi_2 ) \left( \xi,  s \right) = h(\xi,s) - \Theta(s).
\end{equation}
We define  $u^0$ as the solution of the homogenized problem 
that reads
$$
 \partial_t  u^0(x,t) =  \Theta^{{\rm eff}} \nabla \nabla u^0(x,t),\qquad u^0(x,0)=\imath(x),
$$
see also \eqref{eq:homo_pb}. 
Under our standing assumptions $u^0$ and its derivatives in time of any order are  Schwartz class functions in the variable $x$.
Using the definition of corrector $\chi_2$ in \eqref{eq:second_corrector} and Equation \eqref{eq:first_step}  for  $u^0$,  we derive from \eqref{eq:first_step} the following relation:
\begin{align} \label{eq:second_step}
& \partial_t  R^\eps(x,t)  - L^\eps_t R^\eps(x,t)  =  \widetilde \Theta \left( \dfrac{t}{\eps^2}\right) \nabla \nabla u^0(x,t)  \\ \nonumber
& \quad -  \eps \chi_1 \left( \dfrac{x}{\eps}, \dfrac{t}{\eps^2} \right) \partial_t \nabla u^0(x,t)  - \eps^2 \chi_2\left( \dfrac{x}{\eps}, \dfrac{t}{\eps^2} \right) \partial_t \nabla\nabla  u^0(x,t) \\ \nonumber
& \quad + \eps H\left( \dfrac{x}{\eps},\dfrac{t}{\eps^2}\right)  \nabla \nabla \nabla u^0(x , t) + \eps \int_{\R^d} a\left( z\right) \Lambda\left( \dfrac{x}{\eps},  \dfrac{x}{\eps} -z,  \dfrac{t}{\eps^2} \right) G^\eps(x,z,t) dz,
\end{align}
with
\begin{align*}
H(\xi,s)  & =  \int_{\R^d} a\left( z\right) \Lambda\left( \xi,  \xi -z,  s \right)  \times \\
&\hspace{3cm} \left[ - \dfrac{1}{6} (z \otimes z \otimes z)+\dfrac{1}{2} \chi_1 \left( \xi -z, s \right) \otimes z\otimes z)-\chi_2 \left( \xi -z, s\right)\otimes z \right] dz,
\end{align*}
and
\begin{align*}
& G^\eps(x,z,t)  =  - \left(  \int_0^1 \left( \nabla \nabla \nabla u^0(x - \eps q z, t) -  \nabla \nabla \nabla u^0(x , t) \right) \frac12  (1-q)^2 dq \right) (z \otimes z \otimes z) \\ \nonumber
&  \quad  + \chi_1 \left( \dfrac{x}{\eps} -z, \dfrac{t}{\eps^2} \right) \otimes z\otimes z \left( \int_0^1 \left( \nabla \nabla\nabla u^0 \left( x- \eps q z, t\right) -\nabla \nabla \nabla u^0(x , t) \right)  (1-q) dq  \right) \\\nonumber
&\quad  -\chi_2 \left( \dfrac{x}{\eps} -z, \dfrac{t}{\eps^2} \right) \otimes z \left( \int_0^1\left(  \nabla \nabla\nabla u^0\left( x- \eps q z, t\right)-\nabla \nabla \nabla u^0(x , t)\right)  dq \right).
\end{align*}
Note that the initial value for $R^\eps$ is given by:
$$R^\eps(x,0) = - \eps \chi_1 \left( \dfrac{x}{\eps},0\right) \nabla u^0(x,0) - \eps^2 \chi_2\left( \dfrac{x}{\eps}, 0 \right) \nabla\nabla  u^0(x,0).$$
Considering the  linearity of the equation, we split $R^\eps$ into two parts: $R^\eps = \eps \left[R^{\eps,(1)} +R^{\eps,(2)}\right] $ as follows:
\begin{itemize}
\item $R^{\eps,(1)}$ solves:
\begin{equation} \label{eq:stoch_part}
\partial_t  R^{\eps,(1)} (x,t)  - L^\eps_t R^{\eps,(1)} (x,t)  = \dfrac{1}{\eps} \widetilde \Theta \left( \dfrac{t}{\eps^2}\right) \nabla \nabla u^0(x,t) ,
\end{equation}
with zero initial condition.
\item $R^{\eps,(2)} $ solves:
\begin{align} \label{eq:drift_part}
&  \partial_t  R^{\eps,(2)}(x,t)  - L^\eps_t R^{\eps,(2)}(x,t)  = H\left( \dfrac{x}{\eps},\dfrac{t}{\eps^2}\right)  \nabla \nabla \nabla u^0(x , t)   \\ \nonumber
& \quad -  \chi_1 \left( \dfrac{x}{\eps}, \dfrac{t}{\eps^2} \right) \partial_t \nabla u^0(x,t)  - \eps \chi_2\left( \dfrac{x}{\eps}, \dfrac{t}{\eps^2} \right) \partial_t \nabla\nabla  u^0(x,t) , \\ \nonumber
& \quad +  \int_{\R^d} a\left( z\right) \Lambda\left( \dfrac{x}{\eps},  \dfrac{x}{\eps} -z,  \dfrac{t}{\eps^2} \right) G^\eps(x,z,t) dz
\end{align}
with initial condition $R^\eps(x,0)$.
\end{itemize}

\subsection{Convergence of $R^{\eps,(1)} $}

\begin{prop} \label{prop:convergence_random_part}
The process $R^{\eps,(1)} $ converges in law in $L^2(\R^d \times (0,T))$ to the solution of the SPDE:
\begin{equation} \label{eq:limit_SPDE_without_drift}
d v_t -{\rm div }( \Theta^{{\rm eff}} \nabla v ) dt= A^{{\rm eff}} dW_t \nabla\nabla u^0, \quad v(x,0) = 0.
\end{equation}
\end{prop}
The proof of this result is split in the next lemmata.
Observe that $\widetilde \Theta$ given by \eqref{eq:effective_matrix} is stationary and satisfies for any $s$ the relation $\mathbb E \widetilde \Theta(s) = 0$.  Denote
$$\kappa^\eps(t) =\dfrac{1}{\eps} \int_0^t \widetilde \Theta \left( \dfrac{s}{\eps^2} \right) ds.$$

\begin{lemma} \label{lem:CLT}
The process $\kappa^\eps$ converges in law in $C([0,T];\mathbb R^{d^2})$ to $A^{{\rm eff}} W$ where $A^{{\rm eff}}$ is a $d^2\times d^2$-matrix and $W$ a $d^2$-dimensional standard Brownian motion.
\end{lemma}
\begin{proof}
We just have to prove that this process satisfies a functional central limit theorem (invariance principle), which holds if
\begin{equation}\label{eq:mixing_property}
 \int_0^\infty  \mathbb E\widetilde \Theta (s) \widetilde \Theta (0)  ds < \infty ;
 \end{equation}
see \cite[Lemma VIII.3.102]{JaShi}. Assume this property for the moment. Then we apply
 \cite[Theorem VIII.3.79]{JaShi} to deduce that $\kappa^\eps$ converges in law in $C([0,T]; \R^{d\times d})$ to a $d^2$-dimensional Brownian motion with variance coefficient equal to
$$\mathcal C= \mathbb E \int_0^\infty (\widetilde \Theta (s) \widetilde \Theta (0) + \widetilde \Theta (0) \widetilde \Theta (s) ) ds .$$
Note that $\mathcal C$ is a symmetric and positive semi-definite $d^2 \times d^2$ matrix. Thus its square root $A^{{\rm eff}}$ is well-defined and the conclusion of the lemma follows.

Now let us prove \eqref{eq:mixing_property}. In view of  \eqref{eq:effective_matrix}, we need to control the mixing property of the product of $\Lambda$ and $\chi_1$. Let us fix $s > 0$ and define $\widehat \chi_1$ as the solution of: for $r \in [s/2,s]$
$$
\partial_r \widehat \chi_1(\xi,r) - L_r \widehat \chi_1(\xi,r) = g(\xi,r),\quad \widehat \chi_1(\xi,s/2) = 0,
$$
with $g$ defined in \eqref{def_g}.
Note that $ \widehat \chi_1(\cdot,s)$ is $\mathcal F_{\geq s/2}$-measurable. And using \cite[Lemma 3.5]{piat:zhiz:25}, we deduce that there exists $C > 0$ and $\gamma_0 > 0$ s.t.
\begin{equation}\label{appr_chi}
\| \widehat \chi_1(\cdot,r) - \chi_1(\cdot,r) \|_{L^2(\mathbb T^d)} \leq C e^{-\gamma_0 (r-\frac s2)}.
\end{equation}
Thus we have
\begin{align*}
\widetilde \Theta(s) & = \Theta(s) - \Theta^{{\rm eff}} = \int_{\mathbb T^d}  h(\xi,s) d\xi  - \Theta^{{\rm eff}}  \\
 &=   \int_{\mathbb T^d} \int_{\R^d} a\left( z\right) \Lambda\left( \xi,  \xi -z,  s \right)  \left( \dfrac{1}{2} (z \otimes z) - \chi_1 \left( \xi -z, s \right)\otimes z  \right)   dzd\xi  - \Theta^{{\rm eff}} \\
& = \widehat \Theta(s) + \theta(s),
\end{align*}
with
\begin{align*}
\widehat \Theta(s)  & = \int_{\mathbb T^d} \int_{\R^d} a\left( z\right) \Lambda\left( \xi,  \xi -z,  s \right)  \left( \dfrac{1}{2} (z \otimes z) - \widehat \chi_1 \left( \xi -z, s \right)\otimes z  \right)   dzd\xi  - \Theta^{{\rm eff}} , \\
\theta(s)  & =  \int_{\mathbb T^d} \int_{\R^d} a\left( z\right) \Lambda\left( \xi,  \xi -z,  s \right)  \left( \widehat \chi_1 \left( \xi -z, s \right) - \chi_1\left( \xi -z, s \right)  \right)\otimes z  dzd\xi .
\end{align*}
Due to \eqref{appr_chi} the absolute values of the entries of matrix $\theta$ are bounded by $C e^{-\gamma_0 \frac s2}$. 
Then on the right-hand side of the relation
\begin{align*}
\mathbb E ( \widetilde \Theta(s)  \widetilde  \Theta(0))&  = \mathbb E (\widehat \Theta(s) \widetilde \Theta(0)) + \mathbb E (\theta(s) \widehat \Theta(0) )
\end{align*}
the second term is bounded by $C e^{-\gamma_0 \frac s2}$, whereas the first one is bounded by $\rho(s/2)$ (See Condition {\bf H3}).
Therefore,
$$
0\leq \mathbb E ( \widetilde \Theta(s)  \widetilde  \Theta(0)) \leq \rho(s/2) + Ce^{-\gamma_0 s}.
$$
Due to  Condition  {\bf H3} this yields \eqref{eq:mixing_property},  and the statement  of Lemma follows.
\end{proof}

\begin{lemma} \label{lem:approx_sol}
The function  $R^{\eps,(1)} -\kappa^\eps \nabla \nabla u^0$ converges in probability to zero in $L^\infty(0,T; L^2( \R^d))$ as $\eps\to 0$.
\end{lemma}
\begin{proof}
Denote
\begin{align*}
& \widehat R^{\eps,(1)} (x,t)  =R^{\eps,(1)} (x,t) -  \kappa^\eps(t)  \nabla \nabla u^0(x,t) \\
&\qquad  -\eps\kappa^\eps(t) \chi_1\left( \dfrac{x}{\eps}, \dfrac{t}{\eps^2} \right) \nabla \nabla \nabla u^0(x,t) -\eps^2\kappa^\eps(t) \chi_2\left( \dfrac{x}{\eps}, \dfrac{t}{\eps^2} \right) \nabla \nabla \nabla \nabla u^0(x,t).
\end{align*}
Combining \eqref{eq:stoch_part}, the definition \eqref{eq:first_corrector} and \eqref{eq:second_corrector} of $\chi_1$ and $\chi_2$ and equation \eqref{eq:homo_pb} for $u^0$ , we obtain
\begin{align*}
& \partial_t  \widehat R^{\eps,(1)} (x,t)  - L^\eps_t \widehat R^{\eps,(1)} (x,t) \\
& \quad = \kappa^\eps \left(  t \right) \widetilde \Theta \left( \dfrac{t}{\eps^2}\right) \nabla \nabla\nabla\nabla u^0(x,t) - \widetilde \Theta\left( \dfrac{t}{\eps^2}\right)  \chi_1\left( \dfrac{x}{\eps}, \dfrac{t}{\eps^2} \right) \nabla \nabla \nabla u^0(x,t)  \\
& \qquad  - \eps\kappa^\eps(t)  \left[ \chi_1\left( \dfrac{x}{\eps}, \dfrac{t}{\eps^2} \right) \partial_t \nabla \nabla \nabla u^0(x,t)  + \int_{\R^d} a\left( z\right) \Lambda\left( \dfrac{x}{\eps},  \dfrac{x}{\eps} -z,  \dfrac{t}{\eps^2} \right)  \widehat P^\eps (x,z,t) dz \right] \\
& \qquad - \eps \widetilde \Theta \left( \dfrac{t}{\eps^2}\right)  \chi_2\left( \dfrac{x}{\eps}, \dfrac{t}{\eps^2} \right) \nabla \nabla \nabla \nabla u^0(x,t) - \eps^2 \kappa^\eps(t) \chi_2\left( \dfrac{x}{\eps}, \dfrac{t}{\eps^2} \right)  \partial_t \nabla \nabla \nabla \nabla u^0(x,t),
\end{align*}
where
\begin{align*}
\widehat P^\eps (x,z,t) 
& =-  \left(  \int_0^1 \nabla^{(5)}  u^0\left( x- \eps q z, t\right)   (1-q)^2  dq \right) (z \otimes z\otimes z)\\
& -  \chi_2 \left( \dfrac{x}{\eps} -z, \dfrac{t}{\eps^2} \right) \otimes z\left(  \int_0^1  \nabla^{(5)}  u^0\left( x- \eps q z, t\right) dq \right)  \\
&+   \chi_1 \left( \dfrac{x}{\eps} -z, \dfrac{t}{\eps^2} \right) \otimes z\otimes z \left( \int_0^1 \nabla^{(5)} u^0\left( x- \eps q z, t\right) (1-q) dq \right) .
\end{align*}

By linearity, $ \widehat R^{\eps,(1)} $ can be represented as the sum of solutions of  the following three problems:
\begin{enumerate}
\item
The first one reads
\begin{align*}
& \partial_t w_1 (x,t)  - L^\eps_t  w_1 (x,t)  =\eps f^\eps (x,t)\\
&\quad :=  - \eps\kappa^\eps(t)  \left[ \chi_1\left( \dfrac{x}{\eps}, \dfrac{t}{\eps^2} \right) \partial_t \nabla \nabla \nabla u^0(x,t)  + \int_{\R^d} a\left( z\right) \Lambda\left( \dfrac{x}{\eps},  \dfrac{x}{\eps} -z,  \dfrac{t}{\eps^2} \right)  \widehat P^\eps (x,z,t) dz \right] \\
& \qquad - \eps \widetilde \Theta \left( \dfrac{t}{\eps^2}\right)  \chi_2\left( \dfrac{x}{\eps}, \dfrac{t}{\eps^2} \right) \nabla \nabla \nabla \nabla u^0(x,t) - \eps^2 \kappa^\eps(t) \chi_2\left( \dfrac{x}{\eps}, \dfrac{t}{\eps^2} \right)  \partial_t \nabla \nabla \nabla \nabla u^0(x,t),\\
&\qquad \qquad w_1(x,0)=0.
\end{align*}
By the Birkhoff ergodic theorem, for any $t\in [0,T]$  a.s. $\eps\kappa^\eps(t)$ converges to zero. Considering the uniform
boundedness of $\widetilde\Theta(\cdot)$ we conclude that a.s. $|\eps\kappa^\eps(\cdot)|$ converges to zero
in $C([0,T])$. This implies that the first term on the right-hand side of the last equation a.s. tends to zero
in $L^\infty(0,T;L^2(\mathbb R^d))$ as $\eps\to0$.  The norm of two other terms in this space does not exceed
$C\eps$  with a deterministic constant $C$. Therefore,  a.s.
\begin{equation}\label{eqeqeq_a}
\|w_1\|_{L^\infty(0,T;L^2(\R^d))} \leq C \eps \| f^\eps \|_{L^\infty(0,T;L^2(\R^d))} \longrightarrow 0,\quad\hbox{as }
\eps\to0.
\end{equation}

\item The second equation takes the form
\begin{equation}\label{eq_w2_pbm}
\begin{array}{c}
\displaystyle
\partial_t w_2 (x,t)  - L^\eps_t  w_2 (x,t)  = - \widetilde \Theta\left( \dfrac{t}{\eps^2}\right)  \chi_1\left( \dfrac{x}{\eps}, \dfrac{t}{\eps^2} \right) \nabla \nabla \nabla u^0(x,t),\\
w_2(x,0)=0.
\end{array}
\end{equation}
\begin{lemma}\label{lem_3_i}
We have
\begin{equation}\label{ineq_lem_w2}
\|w_2\|_{L^\infty(0,T; H^1(\mathbb R^d))}\leqslant C\eps.
\end{equation}
\end{lemma}
\begin{proof}
Consider the equation
\begin{equation} \label{eq:aux_w2}
\partial_s \Xi (\xi, s) - (L_s \Xi) \left( \xi,  s \right) = - \widetilde \Theta(s)  \chi_1(\xi, s).
\end{equation}
By construction, the functions $\widetilde\Theta(s)$ and $\chi_1(\xi,s)$ on the right-hand side are such that
$$
\int_{\mathbb T^d}  \widetilde \Theta\left(  s \right)  \chi_1\left( \xi,  s \right) d\xi = 0\quad\hbox{for all s,}\quad
\hbox{and }\ \| \Theta  \chi_1 \|_{L^\infty(-\infty,+\infty;L^2(\mathbb T^d))}\leqslant C.
$$
Then due to \cite[Lemma 3.8]{piat:zhiz:25}, Equation \eqref{eq:aux_w2} has a unique periodic in $\xi$ and stationary in $s$ solution $\Xi \in L^\infty(-\infty,\infty;L^2(\mathbb T^d))$, and
 the function $\tilde w_2:= \eps^2 \Xi\big(\frac x\eps, \frac t{\eps^2}\big)\nabla\nabla\nabla u^0(x,t)$ satisfies the equation
\begin{equation}\label{eq_tilde_w2}
\begin{array}{c}
\displaystyle
\partial_t \tilde w_2 (x,t)  - L^\eps_t  \tilde w_2 (x,t)  = - \widetilde \Theta\Big( \dfrac{t}{\eps^2}\Big)  \chi_1\Big( \dfrac{x}{\eps}, \dfrac{t}{\eps^2}\Big) \nabla \nabla \nabla u^0(x,t)\\[3mm]
\displaystyle
+\eps^2\Xi\Big(\frac x\eps, \frac t{\eps^2}\Big)\nabla\nabla\nabla \partial_t u^0(x,t)\\[3mm]
\displaystyle
+\eps\int_{\mathbb R^d}dz\int_0^1d\tau a(z)\Lambda\Big(\frac x\eps, \frac x\eps-z,\frac t{\eps^2}\Big)
\Xi\Big(\frac x\eps-z,\frac t{\eps^2}\Big)
\nabla(\nabla\nabla\nabla) u^0(x-\tau z,t)\cdot z.
\end{array}
\end{equation}
Denote  $\tilde v_{2,\eps}(x,z,\tau,t)=\Xi\big(\frac x\eps-z,\frac t{\eps^2}\big)
\nabla(\nabla\nabla\nabla) u^0(x-\tau z,t)$. Since $u^0$  and its derivatives in time belong to the Schwartz class,
and  $\Xi \in L^\infty(-\infty,\infty;L^2(\mathbb T^d))$, we have
\begin{equation}\label{estima_w2}
\|\tilde v_{2,\eps}(\cdot,z,\tau,t)\|_{L^2(\mathbb R^d)}\leqslant C,\quad \|\tilde w_{2}(\cdot,t)\|_{L^2(\mathbb R^d)}\leqslant C\eps^2,
\end{equation}
where the constant $C$ is independent of $\eps$, $t$, $\tau$ and $z$. Therefore,
$$
\begin{array}{c}
\displaystyle
\int_{\mathbb R^d}dx\bigg(\int_{\mathbb R^d}dz\int_0^1d\tau a(z)\Lambda\Big(\frac x\eps, \frac x\eps-z,\frac t{\eps^2}\Big)
\Xi\Big(\frac x\eps-z,\frac t{\eps^2}\Big)
\nabla(\nabla\nabla\nabla) u^0(x-\tau z,t)\cdot z\bigg)^2\\[3mm]
\displaystyle
\leqslant \!(\Lambda^+)^2\!\!\int\limits_{\mathbb R^d}a(z)|z|dz\int\limits_{\mathbb R^d}a(y)|y|dy\int\limits_0^1d\tau \!
\int\limits_{\mathbb R^d}|\tilde v_{2,\eps}(x,z,\tau,t)|\, |\tilde v_{2,\eps}(x,y,\tau,t)|dx\leqslant  (\Lambda^+)^2M_1^2C^2,
\end{array}
$$
here the symbol $M_1$ stands for $\int_{\mathbb R^d}a(z)|z|dz$.
Taking the difference of Equations \eqref{eq_w2_pbm} and \eqref{eq_tilde_w2} and using \cite[Proposition 2]{piat:zhiz:23a} and the last estimate, we conclude that
$$
\|w_2-\tilde w_2\|_{L^\infty(0,T; L^2(\mathbb R^d))}\leqslant C\eps.
$$
Combining this inequality with the second inequality in \eqref{estima_w2} yields the desired estimate \eqref{ineq_lem_w2}.
\end{proof}


\item The only remaining term is the solution of the third problem
\begin{equation}\label{third_pbm}
\partial_t w_3 (x,t)  - L^\eps_t  w_3 (x,t)  = \kappa^\eps \left(  t \right) \widetilde \Theta \left( \dfrac{t}{\eps^2}\right) \nabla^{(4)} u^0(x,t), \ \ w_3(x,0)=0.
\end{equation}
Observe that
$$
\int_0^t \kappa^\eps \left( s \right) \widetilde \Theta \left( \dfrac{s}{\eps^2}\right) ds =\frac\eps2 (\kappa^{\eps}(t))^2.
$$
By Lemma \ref{lem:CLT} the process $\eps(\kappa^\eps(\cdot))$ converges to zero in probability in the space
$C([0,T])$, so does the process $\eps(\kappa^\eps(\cdot))^2$.  Letting
$$
\tilde w_3(x,t)=\frac12\eps(\kappa^\eps(t))^2 \Big[\nabla^{(4)} u^0(x,t)+\eps\chi_1\Big(
 \dfrac{x}{\eps}, \dfrac{t}{\eps^2}\Big)\nabla(\nabla^{(4)} u^0(x,t))\Big]
$$
and making straightforward computations we conclude that the function $\tilde v_{3}:=\tilde w_3-w_3$ satisfies the following problem:
\begin{align}\label{third_pbm_aux}
& \partial_t \tilde v_{3} (x,t)  - L^\eps_t  \tilde v_3 (x,t)  =\eps \kappa^\eps(t ) \widetilde \Theta \Big( \dfrac{t}{\eps^2}\Big) 
\chi_1\Big( \dfrac{x}{\eps}, \dfrac{t}{\eps^2}\Big)\nabla(\nabla^{(4)} u^0(x,t))
 \\ \nonumber & \quad
 +\frac12\eps(\kappa^\eps(t))^2\Big[\nabla^{(4)}\partial_t u^0(x,t) + \eps\chi_1\Big(
 \dfrac{x}{\eps}, \dfrac{t}{\eps^2}\Big)\nabla(\nabla^{(4)} \partial_t u^0(x,t))\Big]
  \\ \nonumber & \quad +\frac12\eps(\kappa^\eps(t))^2\!\! \int\limits_{\mathbb R^d}\! a(z)\Lambda\Big(\frac x\eps, \frac x\eps\! -\! z,\frac t{\eps^2}\Big)\Big[
 z\otimes z \!\int\limits_0^1\!\nabla\nabla (\nabla^{(4)} u^0(x-\eps \tau(1\!-\!\tau)z,t))d\tau
 \\ \nonumber & \quad +z\otimes \chi_1\Big( \dfrac{x}{\eps}, \dfrac{t}{\eps^2}\Big)
 \int\limits_0^1 \nabla\nabla (\nabla^{(4)} u^0(x-\eps \tau z,t))d\tau\Big]dz,
 \qquad \tilde v_3 (x,0)=0.
\end{align}
In the same way as in the proof of Lemma  \ref{lem_3_i} one can show that the norm of the right-hand side of equation
\eqref{third_pbm_aux} converges to zero in probability  in the space $L^\infty(0,T; L^2(\mathbb R^d))$, as $\eps\to0$.
This implies that $\|\tilde v_3\|_{L^\infty(0,T; L^2(\mathbb R^d))}\to 0$ in probability, as $\eps\to0$.

It follows from the definition of $\tilde w_3$ that $\|\tilde w_3\|_{L^\infty(0,T; L^2(\mathbb R^d))}\to 0$ in probability, as $\eps\to0$. Combining the last two relations we obtain
$$
\|w_3\|_{L^\infty(0,T; L^2(\mathbb R^d))}\to 0\quad\hbox{ in probability, as } \eps\to0.
$$
As a consequence of \eqref{eqeqeq_a}, \eqref{ineq_lem_w2} and the last estimate we have
\begin{equation}\label{est_Reps1}
\|\widehat R^{\eps,(1)} \|_{L^\infty(0,T; L^2(\mathbb R^d))}\to 0\quad\hbox{ in probability, as } \eps\to0.
\end{equation}


\end{enumerate}
From \eqref{est_Reps1}  and the definition of $\widehat R^{\eps,(1)}$  we deduce that  $ R^{\eps,(1)}-
\kappa^\eps \nabla\nabla u^0$
converges to zero in probability in  $L^\infty(0,T; L^2(\mathbb R^d))$.  This completes the proof of Lemma \ref{lem:approx_sol}.
\end{proof}

To complete the proof of Proposition \ref{prop:convergence_random_part} notice that, by Lemmata \ref{lem:CLT} and \ref{lem:approx_sol}, the random function $R^{\eps,(1)}$ converges in law in $L^\infty(0,T;L^2(\mathbb R^d))$
to the process $A^{\rm eff}\nabla\nabla u^0 W_t$, where $W_t$ is the standard $d^2$ dimensional Wiener process.
Applying It\^o's formula to the latter function we conclude that it is a solution of Equation \eqref{eq:limit_SPDE_without_drift}. According to \cite[Theorem 5.4]{dp_zab2014} Equation \eqref{eq:limit_SPDE_without_drift} is well-posed and has exactly one weak solution.

\subsection{Convergence of $R^{\eps,(2)} $}

We introduce the quantities:
\begin{equation} \label{eq:effective_drift}
\langle H \rangle (s) = \int_{\mathbb T^d} H(\xi,s) d\xi, \quad H^{{\rm eff}} = \mathbb E \langle H \rangle (s) .
\end{equation}
The next result is:
\begin{prop} \label{prop:convergence_drift_part}
The function
$R^{\eps,(2)} $ converges a.s. in $L^2(\R^d \times (0,T))$ to a solution of the following problem:
\begin{equation} \label{eq:limit_PDE_with_drift}
\partial_t u -{\rm div }( \Theta^{{\rm eff}} \nabla u )= H^{{\rm eff}} \nabla\nabla\nabla u^0, \quad u(x,0) = 0.
\end{equation}
\end{prop}
In the rest of this section we provide a proof of this statement. We introduce one more corrector $\chi_3$ as a stationary solution of the equation
$$
(\partial_s \chi_3 + L_s \chi_3)(\xi,s) = H(\xi,s) - \langle H \rangle (s).
$$
Since the right-hand side in this equation  has zero mean, the existence of a stationary solution  $\chi_3$ and its uniqueness
up to an additive constant are granted by \cite[Lemma 3.8]{piat:zhiz:25}. 
Similarly, since
$$
\int_{\mathbb T^d} \chi_1(\xi,s) d\xi =0,
$$
the equation $(\partial_s \chi_4 + L_s \chi_4)(\xi,s) = \chi_1(\xi,s)$ has a stationary solution, we denote it by $ \chi_4$.
Letting
$$
 \widehat R^{\eps,(2)}(x,t)= R^{\eps,(2)}(x,t)-\eps^2 \left[ \chi_3 \left( \dfrac{x}{\eps},\dfrac{t}{\eps^2} \right) \nabla\nabla \nabla u^0(x,t)  + \chi_4 \left( \dfrac{x}{\eps},\dfrac{t}{\eps^2} \right) \partial_t \nabla \nabla u^0(x,t) \right]
$$
and considering the definition of $\chi_3$ and $\chi_4$, from  \eqref{eq:drift_part} we deduce the equation
\begin{align} \label{eq:drift_part_2}
&  \partial_t  \widehat R^{\eps,(2)}(x,t)  - L^\eps_t \widehat R^{\eps,(2)}(x,t)  =  \langle H \rangle \left(\dfrac{t}{\eps^2}\right)  \nabla \nabla \nabla u^0(x , t)   \\ \nonumber
& \quad  - \eps \chi_2\left( \dfrac{x}{\eps}, \dfrac{t}{\eps^2} \right) \partial_t \nabla\nabla  u^0(x,t) \\ \nonumber
& \quad - \eps^2 \left[ \chi_3 \left( \dfrac{x}{\eps},\dfrac{t}{\eps^2} \right) \partial_t \nabla\nabla \nabla u^0(x,t)  + \chi_4 \left( \dfrac{x}{\eps},\dfrac{t}{\eps^2} \right) \partial^2_t \nabla \nabla u^0(x,t) \right]  \\ \nonumber
& \quad + \int_{\R^d} a\left( z\right) \Lambda\left( \dfrac{x}{\eps},  \dfrac{x}{\eps} -z,  \dfrac{t}{\eps^2} \right) \left[ G^\eps(x,z,t) + \eps \widehat G^\eps(x,z,t) z \right]dz,
\end{align}
where
\begin{align*}
& \widehat G^\eps(x,z,t)  = - \chi_3 \left( \dfrac{x}{\eps} -z, \dfrac{t}{\eps^2} \right) \left( \int_0^1 \left( \nabla \nabla\nabla u^0 \left( x- \eps q z, t\right) -\nabla \nabla \nabla u^0(x , t) \right)  dq  \right)  \\\nonumber
&\quad  +\chi_4 \left( \dfrac{x}{\eps} -z, \dfrac{t}{\eps^2} \right) \left( \int_0^1\left(  \partial_t \nabla\nabla u^0\left( x- \eps q z, t\right)-\partial_t \nabla \nabla u^0(x , t)\right)  dq \right)  .
\end{align*}
We split $\widehat R^{\eps,(2)}$ into three parts: $\widehat R^{\eps,(2)} = r^{\eps,(1)} + r^{\eps,(2)} + r^{\eps,(3)}$, where
\begin{itemize}
\item $r^{\eps,(1)}$ satisfies the equation
\begin{equation} \label{eq:drift_part_3}
\partial_t  r^{\eps,(1)} (x,t)  - L^\eps_t r^{\eps,(1)} (x,t)  =   \langle H \rangle \left(\dfrac{t}{\eps^2}\right)  \nabla \nabla \nabla u^0(x , t)
\end{equation}
with initial condition zero;
\item $r^{\eps,(2)} $ satisfies the equation
$$
\partial_t  r^{\eps,(2)}(x,t)  - L^\eps_t r^{\eps,(2)}(x,t)  = 0,
$$
with initial condition: $r^{\eps,(2)}(x,0) = -  \chi_1\left( \dfrac{x}{\eps}, 0 \right) \nabla u^0(x,0) $;
\item $r^{\eps,(3)} $ satisfies the equation
\begin{align} \label{eq:neg_part}
&  \partial_t  r^{\eps,(3)}(x,t)  - L^\eps_t r^{\eps,(3)}(x,t)  = - \eps \chi_2\left( \dfrac{x}{\eps}, \dfrac{t}{\eps^2} \right) \partial_t \nabla\nabla  u^0(x,t) \\ \nonumber
& \quad - \eps^2 \left[ \chi_3 \left( \dfrac{x}{\eps},\dfrac{t}{\eps^2} \right) \partial_t \nabla\nabla \nabla u^0(x,t)  + \chi_4 \left( \dfrac{x}{\eps},\dfrac{t}{\eps^2} \right) \partial^2_t \nabla \nabla u^0(x,t) \right]  \\ \nonumber
& \quad +  \int_{\R^d} a\left( z\right) \Lambda\left( \dfrac{x}{\eps},  \dfrac{x}{\eps} -z,  \dfrac{t}{\eps^2} \right) \left[ G^\eps(x,z,t) + \eps \widehat G^\eps(x,z,t) z \right]dz
\end{align}
with initial condition
\begin{align}\label{init_r3}
r^{\eps,(3)}(x,0) & = - \eps \chi_2\left( \dfrac{x}{\eps}, 0 \right) \nabla\nabla  u^0(x,0) \\
\nonumber
& -\eps^2 \left[ \chi_3 \left( \dfrac{x}{\eps},0 \right) \nabla\nabla \nabla u^0(x,0)  + \chi_4 \left( \dfrac{x}{\eps}, 0 \right) \partial_t \nabla \nabla u^0(x,0) \right] .
\end{align}
\end{itemize}

\begin{lemma}
$r^{\eps,(2)}$ converges to zero a.s. in $L^\infty(\R^d\times[0,T])$.
\end{lemma}
\begin{proof}
Define $\mathcal I=\mathcal I(\xi,s)$ as a solution of the following Cauchy problem:
$$
\partial_s  \mathcal I(\xi ,s) - L_s \mathcal I (\xi ,s) = 0\quad\hbox{in }\mathbb T^d\times(0,+\infty),\qquad
\mathcal I(\xi ,0)= -\chi_1(\xi,0).
$$
Since $\int_{\mathbb T^d} \mathcal I (\xi ,0)  d\xi =-\int_{\mathbb T^d} \chi_1 (\xi ,0)   =0$, then, according to \cite[Lemma 3.5]{piat:zhiz:25}, there exists a constant $\nu > 0$ such that
\begin{equation}\label{exp_auxx}
\| \mathcal I (\cdot, s) \|_{L^2(\mathbb T^d)} \leq C e^{-\nu s}.
\end{equation}
It is straightforward to check that the difference $\theta^{2,\eps}(x,t):=r^{\eps,(2)}(x,t)- \mathcal I \big(\frac x\eps, \frac t{\eps^2}\big)\nabla u^0(x,t)$ satisfies the equation
\begin{align}  \nonumber
& \partial_t  \theta^{2,\eps}(x,t)  - L^\eps_t\theta^{2,\eps}(x,t) \\ \label{2_auxili}
& \quad
 =\int_{\mathbb R^d}\eps^{-1}a(z)
\Lambda\left( \dfrac{x}{\eps},  \dfrac{x}{\eps} -z,  \dfrac{t}{\eps^2} \right)\mathcal{I}\Big(\frac x\eps-z,\frac t{\eps^2}\Big) z\int_0^1\nabla\nabla u^0(x-\eps \tau z,t)\,d\tau dz.
\end{align}
 Taking into account Estimate \eqref{exp_auxx}, Condition {\bf H2} and the fact that $u^0(\cdot,t)$ is a Schwartz class function, we conclude that the function
 $$
 \Phi_z^{2,\eps}(x,t):=\Lambda\left( \dfrac{x}{\eps},  \dfrac{x}{\eps} -z,  \dfrac{t}{\eps^2} \right)\mathcal{I}\Big(\frac x\eps-z,\frac t{\eps^2}\Big) \int_0^1\nabla\nabla u^0(x-\eps \tau z,t)\,d\tau
 $$
 satisfies the upper bound
 $$
 \|\Phi_z^{2,\eps}(\cdot,t)\|_{L^2(\mathbb R^d)}\leqslant C e^{-\frac {\nu t}{\eps^2}},
 $$
 where the constant $C$ does not depend on $z$. Therefore, the $L^2$-norm of the right-hand side in
 \eqref{2_auxili} admits the estimate
 $$
 \int_{\mathbb R^d}\Big(\int_{\mathbb R^d}\eps^{-1}a(z) z \,\Phi_z^{2,\eps}(x,t)\,dz\Big)^2\,dx
 =\eps^{-2}\int_{\mathbb R^d}\int_{\mathbb R^d}a(z)z\,a(q)q\int_{\mathbb R^d} \Phi_z^{2,\eps}(x,t)\Phi_q^{2,\eps}(x,t)\,dx
 $$
 $$
 \leqslant \eps^{-2}\int_{\mathbb R^d}\int_{\mathbb R^d}a(z)z\,a(q)q \|\Phi_z^{2,\eps}(\cdot,t)\|_{L^2(\mathbb R^d)}\|\Phi_q^{2,\eps}(\cdot,t)\|_{L^2(\mathbb R^d)}\,dzdq\leqslant C \eps^{-2} e^{-\frac {2\nu t}{\eps^2}}
 $$
 with a constant C that does not depend on $\eps$. According to \cite[Proposition 4.1]{piat:zhiz:25} this yields
 $$
 \|\theta^{2,\eps}\|_{L^\infty(0,T;L^2(\mathbb R^d))}\leqslant C.
 $$
 Multiplying Equation \eqref{2_auxili} by $\theta^{2,\eps}$ and integrating the resulting relation over
 $\mathbb R^d\times[0,T]$ and considering the last two estimates we obtain
\begin{align*}
 \int_{\mathbb R^d} (\theta^{2,\eps}(x,t))^2\,dx & \leqslant C\int_0^t\int_{\mathbb R^d}\theta^{2,\eps}(x,s)
 \int_{\mathbb R^d}\eps^{-1}a(z) z \,\Phi_z^{2,\eps}(x,t)\,dz dxds \\
 &  \leqslant C\int_0^t\|\theta^{2,\eps}(\cdot,s)\|_{L^2(\mathbb R^d)}
 \Big[\int_{\mathbb R^d}\Big(\int_{\mathbb R^d}\eps^{-1}a(z) z \,\Phi_z^{2,\eps}(x,s)\,dz\Big)^2 dx\Big]^\frac12ds \\
 & \leqslant C \int_0^t C \eps^{-1}e^{-\frac{\nu s}{\eps^2}}\leqslant C\eps.
 \end{align*}
 Combining this inequality with \eqref{exp_auxx} yields the upper bound $\|r^{\eps,(2)}\|_{L^\infty(0,T;L^2(\mathbb R^d))}\leqslant C\eps^{\frac12}$ and thus the desired convergence.
\end{proof}

\begin{lemma}
$r^{\eps,(3)}$ converges to zero a.s. in $L^\infty(0,T;L^2(\R^d))$.
\end{lemma}
\begin{proof}
Our goal is to estimate the $L^2$ norm of the  term
$$
 \int_{\R^d} a\left( z\right) \Lambda\left( \dfrac{x}{\eps},  \dfrac{x}{\eps} -z,  \dfrac{t}{\eps^2} \right) G^\eps(x,z,t) dz
$$
on the right-hand side of \eqref{eq:neg_part}. The smallness of the other terms is clear.
\\
From the boundedness of $\Lambda$, $\chi_1$, $\chi_2$ and $\nabla\nabla\nabla u^0$, we obtain that
$| G^\eps(x,z,t) | \leq C(1+ |z|^3).$
Representing $\mathbb R^d$ as $\mathbb R^d=\{z\,:\,\eps^{\frac12}|z|<1\}\cup \{z\,:\,\eps^{\frac12}|z|\geqslant 1\}$ we have
\begin{align*}
&  \int_{\mathbb R^d}\Big(\int_{\sqrt{\eps} |z| \geq 1} a\left( z\right) \Lambda\left( \dfrac{x}{\eps},  \dfrac{x}{\eps} -z,  \dfrac{t}{\eps^2} \right) |G^\eps(x,z,t) |dz\Big)^2\,dx \\
& \quad  \leqslant  C \Big(\int_{\sqrt{\eps} |z| \geq 1} a\left( z\right) (1+|z|^3) dz\Big)^2 \,
   \|\nabla\nabla\nabla u^0(\cdot,t)\|^2_{L^2(\mathbb R^d)} \\
   & \quad \leqslant C \Big(\int_{\sqrt{\eps} |z| \geq 1} a\left( z\right) (1+|z|^3) dz\Big)^2 .
 \end{align*}
Since the function $a(z)(1+|z|^3)$ is integrable, the right-hand side here tends to zero as $\eps\to0$.
If $z\in\{z\,:\,\sqrt{\eps} |z| < 1\}$, then for any $q \in [0,1]$ the following estimate holds:
$$
 \left| \nabla \nabla\nabla u^0 \left( x- \eps  q z, t\right) -\nabla \nabla \nabla u^0(x , t) \right| \leqslant
 \eps^{\frac12}\|\nabla\nabla\nabla\nabla u^0(\cdot,t)\|_{C(B_1(x))},
 $$
 where $B_1(x)=\{y\in\mathbb R^d\,:\,|y-x|\leqslant 1\}$.
Hence
\begin{align*}
& \int_{\mathbb R^d}\Big(\int_{\sqrt{\eps} |z| < 1} a\left( z\right) \Lambda\left( \dfrac{x}{\eps},  \dfrac{x}{\eps} -z,  \dfrac{t}{\eps^2} \right) |G^\eps(x,z,t) |dz\Big)^2\,dx
\\
& \quad \leqslant C\eps \Big(\int_{\mathbb R^d} a\left( z\right) (1+|z|^3) dz\Big)^2\,\int_{\mathbb R^d}
\|\nabla\nabla\nabla\nabla u^0(\cdot,t)\|^2_{C(B_1(x))}\,dx\leqslant  C\eps.
\end{align*}
Since the initial condition in \eqref{init_r3} satisfies the estimate $\|r^{\eps,(3)}(\cdot,0)\|_{L^2(\mathbb R^d)}
\leqslant c\eps$, by \cite[Proposition 4.1]{piat:zhiz:25} 
we conclude that there exists a constant $C$ such that
$$
\|r^{\eps,(3)}\|^2_{L^\infty(0,T;L^2(\mathbb R^d))}\leqslant
 C\Big(\int_{\sqrt{\eps} |z| \geq 1} a\left( z\right) (1+|z|^3) dz\Big)^2+ C\eps,
$$
which implies the required convergence.
\end{proof}

\begin{lemma} \label{lem:conv_drift_part}
$r^{\eps,(1)}$ converges a.s. in $L^2(\R^d)$ to the unique solution of \eqref{eq:limit_PDE_with_drift}.
\end{lemma}
\begin{proof}
Let us split $r^{\eps,(1)}$ into two parts  $\rho^\eps$ and $\varrho^\eps$,  $r^{\eps,(1)}=\rho^\eps + \varrho^\eps$, such that
\begin{itemize}
\item $\rho^\eps$ solves
\begin{equation}\label{eq_rho_plain}
\partial_t  \rho^{\eps} (x,t)  - L^\eps_t \rho^{\eps} (x,t)  =  \left(  \langle H \rangle \left(\dfrac{t}{\eps^2}\right) - H^{{\rm eff}} \right)  \nabla \nabla \nabla u^0(x , t)
\end{equation}
\item $\varrho^\eps$ solves
\begin{equation}\label{eq_rho_var}
\partial_t  \varrho^{\eps} (x,t)  - L^\eps_t \varrho^{\eps} (x,t)  =  H^{{\rm eff}}  \nabla \nabla \nabla u^0(x , t)
\end{equation}
\end{itemize}
both with zero initial condition. From \cite[Theorem 2.1 ]{piat:zhiz:25}, it follows that Problem
\eqref{eq_rho_var} admits homogenization. In particular, $\varrho^\eps$ converges a.s. in $L^2(\R^d \times (0,T))$ to the solution of \eqref{eq:limit_PDE_with_drift}. In order to show that  $\rho^\eps$ converges a.s. to zero in $L^2(\R^d \times (0,T))$  
we construct the ansatz
$$
\widetilde{\rho}^\eps=\vartheta^\eps(t)\big[\nabla\nabla\nabla u^0(x,t)+\eps\chi_1\Big(\frac x\eps,\frac t{\eps^2}\Big)\nabla\nabla\nabla\nabla u^0(x,t)\big],\quad
 \vartheta^\eps(t)=\int_0^t \Big(\langle H \rangle \left(\dfrac{s}{\eps^2}\right) - H^{{\rm eff}}\Big)ds.
$$
One can easily check that a.s.
$$
\big\|\partial_t(\widetilde{\rho}^\eps-\rho^\eps)-L_t^\eps(\widetilde{\rho}^\eps-\rho^\eps)\big\|_{L^2(\mathbb R^d\times[0,T])}
\to 0,\quad\hbox{as }\eps\to0.
$$
Therefore, $\|\widetilde{\rho}^\eps-\rho^\eps\|_{L^\infty(0,T;L^2(\mathbb R^d))}\to 0$ a.s.  Since by construction
$\|\widetilde{\rho}^\eps\|_{L^\infty(0,T;L^2(\mathbb R^d))}\to 0$ a.s.,  we conclude that
$\|{\rho}^\eps\|_{L^\infty(0,T;L^2(\mathbb R^d))}\to 0$.
\end{proof}

\medskip
The statement of Proposition \ref{prop:convergence_drift_part} is a consequence of the previous three Lemmata.

\subsection{Proof of Theorem \ref{thm:main_result}}

The proof now follows immediately from \eqref{eq:expansion}, the representation $R^\eps = \eps \left[R^{\eps,(1)} +R^{\eps,(2)}\right] $ and Propositions \ref{prop:convergence_random_part} and \ref{prop:convergence_drift_part}.

\section{The non-symmetric case}\label{s_no_sy}

In this part we do not assume that Hypothesis {\bf H5} holds. From \cite{piat:zhiz:25}, it is known that, in contrast with the symmetric case, in general the homogenized problem is stochastic. Moreover, the convergence is valid in moving coordinates. Let us explain the main idea of construction of the moving coordinates. We again start with the expansion \eqref{eq:expansion}:
\begin{align} \label{eq:expansion_2}
u^\eps(x,t) & = u^0(x^\eps,t) + \eps \chi_1 \left( \dfrac{x}{\eps}, \dfrac{t}{\eps^2} \right) \nabla u^0(x^\eps,t) \\ \nonumber
& + \eps^2 \chi_2\left( \dfrac{x}{\eps}, \dfrac{t}{\eps^2} \right) \nabla\nabla  u^0(x^\eps,t) + R^\eps(x,t)
\end{align}
where $x^\eps = x -\dfrac{1}{\eps} \mathcal{G}^\eps(t)$ is the moving coordinates. The process $\mathcal{G}^\eps$ is given by
$$
\mathcal{G}^\eps(t) = \int_0^t \beta\left( \dfrac{s}{\eps^2}\right) ds,
$$
where the stationary process $\beta(s)$ will be defined later. Applying  the operator $\partial_t - L^\eps_t$ to this expansion of $u^\eps$ and collecting the terms of order $\eps^{-1}$ we arrive at the equation which is similar to equation \eqref{eq:first_corrector} up to the term $\beta$, it reads
\begin{equation} \label{eq:first_corrector_bis}
\partial_s \chi_1 (\xi, s) - (L_s \chi_1 ) \left( \xi,  s \right) = -\int_{\R^d} a\left( z\right) \Lambda\left( \xi,  \xi -z,  s \right) z dz + \beta\left( s \right).
\end{equation}
According to \cite[Lemma 3.8]{piat:zhiz:25}, this equation has a stationary solution  if
\begin{equation} \label{eq:def_beta}
\beta(s) =\int_{\mathbb T^d} \left(  \int_{\R^d} a\left( z\right) \Lambda\left( \xi,  \xi -z,  s \right) z dz \right) p(\xi,s) d\xi.
\end{equation}
Here and in the rest of this part, $p$ denotes the unique stationary solution of the equation $-\partial_t p = L^*_t p$ such that
$$
\int_{\mathbb T^d} p(\xi,s) d\xi = 1,
$$
where $L^*$ is the adjoint to $L$ operator.
The existence and uniqueness of  such a solution are granted by \cite[Proposition 3.2]{piat:zhiz:25}. 
Moreover, there exist two constants $0< \pi_1 \leq \pi_2$ such that $\pi_1 \leq p(\xi,s) \leq \pi_2$ for any $(\xi ,s)$,
see \cite[Proposition 3.2]{piat:zhiz:25}. The pair $(\Lambda(\xi,\eta,s), p(\xi,s))$ is also stationary. Notice that  in the symmetric case $p(\xi,s) = 1$.

 The function $\chi_1$ is uniquely defined up to an additive constant. From now on we assume that
 \begin{equation}\label{normaliz_chi1}
 \int_{\mathbb T^d}\chi_1(\xi,s)p(\xi,s)d\xi=0.
 \end{equation}
Let us show that a.s. the integral on the left-hand side does not depend on $s$. Indeed,
differentiating this integral in $s$ yields
$$
\partial_s\int_{\mathbb T^d}\chi_1(\xi,s)p(\xi,s)d\xi=\int_{\mathbb T^d}(L_s\chi_1(\xi,s))p(\xi,s)d\xi
$$
$$
-\int_{\mathbb T^d}\bigg(\int_{\R^d} a\left( z\right) \Lambda\left( \xi,  \xi -z,  s \right) z dz\bigg) p(\xi,s)d\xi
+\beta(s)-\int_{\mathbb T^d}\chi_1(\xi,s) L_s^\ast p(\xi,s)d\xi=0.
$$
Therefore, normalization condition \eqref{normaliz_chi1} is well-defined.

By stationarity, the process $\beta(s)$ has a constant mean value $\bar \beta = \mathbb E \beta(s)$. Then
$$\dfrac{1}{\eps}  \mathcal{G}^\eps(t)  = \dfrac{1}{\eps}  \int_0^t  \left[ \beta\left( \dfrac{s}{\eps^2}\right) -\bar \beta  \right] ds + \dfrac{\bar \beta}{\eps} t = \mathcal{G}^\eps_0(t) +  \dfrac{\bar \beta}{\eps} t  .$$
Under suitable mixing condition, $\mathcal{G}^\eps_0$ converges to $\sigma W_t$, where $W_\cdot$ is a standard Brownian motion (see \cite[Lemma 5.1]{piat:zhiz:25}). In other words, the process $x^\eps - \dfrac{\bar \beta }{\eps} t$ defining moving coordinates converges to the  stochastic process  $(x+\sigma W_t)$, and the effective dynamics remains random.

To have a deterministic homogenized problem, from now on we assume that:
\begin{itemize}
\item[{\bf H6}] The process $\beta$ is deterministic.
\end{itemize}
By stationary, $\beta$ does not depend on time and $\beta = \bar \beta$.
\begin{remark}
Condition {\bf H6} holds in the following cases:
\begin{itemize}
\item If {\bf H5} holds (symmetric case), then $\beta(s) = 0$ for any $s$.
\item If $\Lambda(x,y,s) = \lambda(x,y) \mu(s)$, with
$$\int_{\mathbb T^d} \left(  \int_{\R^d} a\left( z\right) \lambda\left( \xi,  \xi -z \right) z dz \right) p(\xi) d\xi = 0,$$
where $p\in L^2(\mathbb T^d)$, $p\not=0$,  solves the equation:
$$\int_{\R^d} a\left( y-x \right) \lambda\left( y,  x \right)  p(y) dy - \left(  \int_{\R^d} a\left( x-y \right) \lambda\left( x, y \right) dy \right)  p(x) = 0.$$
Then again $\beta(s) = 0$ for any $s$.
\end{itemize}
\end{remark}

Hence we suppose that \eqref{eq:expansion_2} holds with $x^\eps = x -\dfrac{\beta}{\eps} t$. Substituting the right-hand side of \eqref{eq:expansion_2} for $u^\eps$ in the original equation in \eqref{eq:u^eps}
and assuming that $u^0\in C^\infty(0,T;\mathcal{S}(\mathbb R^d))$, in the same way as in the symmetric case we obtain
\begin{align} \label{eq:first_step_nonsym}
& \partial_t  u^0(x^\eps,t) + \partial_t  R^\eps(x,t) + \partial_s \chi_2\Big( \dfrac{x}{\eps}, \dfrac{t}{\eps^2} \Big) \nabla\nabla  u^0(x^\eps,t) - \chi_1 \Big( \dfrac{x}{\eps}, \dfrac{t}{\eps^2} \Big)\otimes\beta\nabla \nabla u^0(x^\eps,t) \\  \nonumber
& \ \  +  \eps \chi_1 \Big( \dfrac{x}{\eps}, \dfrac{t}{\eps^2} \Big) \partial_t \nabla u^0(x^\eps,t) + \eps^2 \chi_2\Big( \dfrac{x}{\eps}, \dfrac{t}{\eps^2} \Big) \partial_t \nabla\nabla  u^0(x^\eps,t)
-\eps \chi_2\Big( \dfrac{x}{\eps}, \dfrac{t}{\eps^2} \Big)\otimes\beta \nabla \nabla\nabla  u^0(x^\eps,t)  \\ \nonumber
& =  L^\eps_t R^\eps(x,t) +  \frac12\int_{\R^d} a( z) \Lambda\Big( \dfrac{x}{\eps},  \dfrac{x}{\eps} -z,  \dfrac{t}{\eps^2} \Big)(z \otimes z) dz\, \nabla \nabla u^0(x^\eps, t) \\ \nonumber
& - \frac\eps2\int_{\R^d} a( z) \Lambda\Big( \dfrac{x}{\eps},  \dfrac{x}{\eps} -z,  \dfrac{t}{\eps^2} \Big)(z \otimes z\otimes z)
\int_0^1 \nabla\nabla\nabla u^0(x^\eps-\eps qz, t)(1-q)^2dq dz \\ \nonumber
&\quad  -  \int_{\R^d} a( z) \Lambda\Big( \dfrac{x}{\eps},  \dfrac{x}{\eps} -z,  \dfrac{t}{\eps^2} \Big) \chi_1 \Big( \dfrac{x}{\eps} -z, \dfrac{t}{\eps^2} \Big)\otimes z dz\,\nabla\nabla u^0 ( x^\eps, t)   \\ \nonumber
&\quad  - \eps \int_{\R^d} a\left( z\right) \Lambda\left( \dfrac{x}{\eps},  \dfrac{x}{\eps} -z,  \dfrac{t}{\eps^2} \right) \chi_1 \left( \dfrac{x}{\eps} -z, \dfrac{t}{\eps^2} \right)\otimes z\otimes z\int_0^1 \nabla\nabla\nabla u^0 \left( x^\eps- \eps q z, t\right)(1-q) dq   dz \\ \nonumber
&\quad  +(L_t \chi_2 ) \Big( \dfrac{x}{\eps},  \dfrac{t}{\eps^2} \Big) \nabla\nabla u^0(x^\eps,t)  \\ \nonumber
&\quad  +\eps \int_{\R^d} a( z) \Lambda\Big( \dfrac{x}{\eps},  \dfrac{x}{\eps} -z,  \dfrac{t}{\eps^2} \Big) \chi_2 \Big( \dfrac{x}{\eps} -z, \dfrac{t}{\eps^2} \Big)\otimes z\int_0^1   \nabla\nabla\nabla u^0( x^\eps- \eps q z, t)  dq  dz;
\end{align}
here we have used the fact that due to the choice of $\chi_1$ the sum of the terms of order $\eps^{-1}$ vanishes.

At the next step we define the matrices
\begin{align} \nonumber
h(\xi,s) & =  \int_{\R^d} a\left( z\right) \Lambda\left( \xi,  \xi -z,  s \right)  \left( \dfrac{1}{2} (z \otimes z) - \chi_1 \left( \xi -z, s \right)\otimes z  \right)   dz +\beta\otimes\chi_1(\xi,s)
\\ \label{eq:effective_matrix_nonsym}
\Theta(s) & = \int_{\mathbb T^d}  h(\xi,s)p(\xi,s) d\xi , \quad \Theta^{{\rm eff,ns}} = \mathbb E (\Theta(s)) \quad \widetilde \Theta(s) = \Theta(s) - \Theta^{{\rm eff,ns}} .
\end{align}
and introduce the function
 $\chi_2(\xi,s)$ as a stationary solution of the equation

\begin{equation} \label{eq:second_corrector_bis}
\partial_s \chi_2 (\xi, s) - (L_t \chi_2 ) \left( \xi,  s \right) = h(\xi,s) - \Theta(s). 
\end{equation}
According to  \cite[Lemma 3.8]{piat:zhiz:25} this equation has a unique up to an additive constant stationary solution $\chi_2(\xi,s)$.  As was shown in \cite[Lemma 1]{piat:zhiz:23a} the matrix $\Theta^{\rm eff,ns}$ is positive definite.
Then we define  $u^0$ as the solution of the homogenized problem that reads
\begin{equation}\label{eff_eq_neso}
 \partial_t  u^0(x,t) =  \Theta^{{\rm eff,ns}} \nabla \nabla u^0(x,t),\qquad u^0(x,0)=\imath(x).
\end{equation}
Note that $u^0$ and all its time derivatives are $C^0(0,T;\mathcal{S}(\mathbb R^d))$.
\begin{lemma}
Let functions $\chi_1$ and $\chi_2$ be given by \eqref{eq:first_corrector_bis} and  \eqref{eq:second_corrector_bis},
respectively.
Then the remainder $R^\eps$ can be represented as $\eps [R^{\eps,(1)} + R^{\eps,(2)}]$, where
\begin{itemize}
\item $R^{\eps,(1)}$  solves  the equation 
\begin{equation} \label{eq:stoch_part_bis}
\partial_t  R^{\eps,(1)} (x,t)  - L^\eps_t R^{\eps,(1)} (x,t)  = \dfrac{1}{\eps} \widetilde \Theta \left( \dfrac{t}{\eps^2}\right) \nabla \nabla u^0(x^\eps,t)
\end{equation}
with initial condition zero;
\item $R^{\eps,(2)} $ solves  the equation that reads 
\begin{align} \label{eq:drift_part_bis}
&  \partial_t  R^{\eps,(2)}(x,t)  - L^\eps_t R^{\eps,(2)}(x,t)  = H\Big( \dfrac{x}{\eps},\dfrac{t}{\eps^2}\Big)  \nabla \nabla \nabla u^0(x^\eps , t)   \\ \nonumber
& \quad -  \chi_1 \left( \dfrac{x}{\eps}, \dfrac{t}{\eps^2} \right) \partial_t \nabla u^0(x^\eps,t)  - \eps \chi_2\left( \dfrac{x}{\eps}, \dfrac{t}{\eps^2} \right) \partial_t \nabla\nabla  u^0(x^\eps,t) \\ \nonumber
& \quad +  \int_{\R^d} a\left( z\right) \Lambda\Big( \dfrac{x}{\eps},  \dfrac{x}{\eps} -z,  \dfrac{t}{\eps^2} \Big) G^\eps(x,z,t) dz -\chi_2\left( \dfrac{x}{\eps}, \dfrac{t}{\eps^2} \right)\otimes\beta \nabla\nabla \nabla  u^0(x^\eps,t),
\end{align}
with the  initial condition $\frac1\eps R^\eps(x,0)$;
here
$$
H\Big( \dfrac{x}{\eps},\dfrac{t}{\eps^2}\Big)=
\int_{\mathbb R^d}a(z)\Lambda\Big( \dfrac{x}{\eps},  \dfrac{x}{\eps} -z,  \dfrac{t}{\eps^2} \Big)
\Big\{\frac16 z\otimes z\otimes z-\frac12\chi_1\Big(\frac x\eps-z,\frac t{\eps^2}\Big)\otimes z\otimes z
$$
$$
+\chi_2\Big(\frac x\eps-z,\frac t{\eps^2}\Big)\otimes z\Big\}dz,
$$
and
\begin{align*}
& G^\eps(x,z,t)  = \frac12z\otimes z\otimes z \int_0^1\big(\nabla\nabla\nabla u^0(x^\eps-\eps q z,t)-
\nabla\nabla\nabla u^0(x^\eps,t)\big)(1-q)^2dq \\
& \quad -\chi_1\Big(\frac x\eps -z,\frac t{\eps^2}\Big)\otimes z\otimes z \int_0^1\big(\nabla\nabla\nabla u^0(x^\eps-\eps q z,t)-
\nabla\nabla\nabla u^0(x^\eps,t)\big)(1-q)dq\\
& \quad +\chi_2\Big(\frac x\eps -z,\frac t{\eps^2}\Big)\otimes z \int_0^1\big(\nabla\nabla\nabla u^0(x^\eps-\eps q z,t)-
\nabla\nabla\nabla u^0(x^\eps,t)\big)dq .
\end{align*}
\end{itemize}
Let us emphasize that in the expressions on the right-hand sides of \eqref{eq:stoch_part_bis} and \eqref{eq:drift_part_bis}, the argument $x$ of the function $u^0$ and its derivatives have been replaced with $x^\eps$.
\end{lemma}

\begin{proof}
Since $\chi_2$ satisfies equation \eqref{eq:second_corrector_bis},  the sum of the terms of order $\eps^0$ in \eqref{eq:first_step_nonsym} is equal to $\partial_t u^0(x,t)-\Theta\big(\frac t{\eps^2}\big)\nabla\nabla u^0(x,t)$.
In view of \eqref{eff_eq_neso} the latter expression is equal to $\widetilde\Theta\big(\frac t{\eps^2}\big)\nabla\nabla u^0(x,t)$.
Letting $R^{\eps,(1)}$ be a solution of \eqref{eq:stoch_part_bis} one can check by direct inspection that
$R^{\eps,(2)}$ is a solution of \eqref{eq:drift_part_bis}.
\end{proof}

The analysis of the remainders $R^{\eps,(1)}$  and $R^{\eps,(2)}$  relies on the following lemmata that generalize the statements of  Propositions \ref{prop:convergence_random_part} and \ref{prop:convergence_drift_part}:
\begin{lemma}\label{l_hom_rhs}
Let $U(x,t)$ be a function such that $U(\cdot,t)$ is from the Schwartz class in $\mathbb R^d$ for each $t\in[0,T]$.
Assume that the time derivatives of $U$ of any order are also Schwartz class functions for all  $t\in[0,T]$. Then
a solution $v^\eps$ of the Cauchy problem
$$
\partial_t v-L_t^\eps v=U(x^\eps,t),\qquad U(x,0)=0,
$$
satisfies a.s. the following limit relation
$$
\lim\limits_{\eps\to0} \|v^\eps(x,t)-v^0\big(x-\frac \beta \eps t,t\big)\|_{L^\infty(0,T;L^2(\mathbb R^d))}=0,
$$
where $v^0$ is a solution to the problem
$$
\partial_t v-\mathrm{div}\big(\Theta^{\rm eff,ns}\nabla v\big)=U(x,t),\qquad U(x,0)=0.
$$
\end{lemma}
\begin{proof}
To obtain the desired convergence one can write down the ansatz
$$
v^\eps(x,t)=v^0(x^\eps, t)+\chi_1\Big(\frac x\eps,\frac t{\eps^2}\Big)\nabla v^0(x^\eps,t)
\chi_2\Big(\frac x\eps,\frac t{\eps^2}\Big)\nabla\nabla v^0(x^\eps,t)+\ldots
$$
and use the standard two-scale expansion arguments of the homogenization theory, see for instance \cite{piat:zhiz:25}. 
We leave the details to the reader.
\end{proof}
\begin{lemma}\label{l_timeosc_small}
Let $\theta(s)$ be a stationary process such that $\mathbf{E}\theta(s)=0$, and assume that the pair $(\theta(s),\Lambda(\xi,\eta,s))$ is also a stationary function of $s$. Assume moreover that $U(x,t)$ and
its time derivative are Schwartz class functions for all $t\in[0,T]$.
Then  a solution $v_1^\eps$ of the Cauchy problem
\begin{equation}\label{eq_lem9}
\partial_t v-L_t^\eps v=\theta\Big(\frac t{\eps^2}\Big)U(x^\eps,t),\qquad v(x,0)=0,
\end{equation}
converges to zero in $L^\infty(0,T;L^2(\mathbb R^d))$.
\end{lemma}
The proof of this and the next Lemma is provided in the Appendix.
\begin{lemma}\label{l_x_osc_small}
Let $\ell(\xi,s)$  be a stationary function of $s$ with values in $L^2(\mathbb T^d)$, and assume that the random
function $(\ell(\xi,s),\Lambda(\xi,\eta,s))$ is also stationary.  Assume that $\int_{\mathbb R^d}\ell(\xi,s)p(\xi,s)d\xi=0$
for all $s\in\mathbb R$.   Assume moreover that $U(x,t)$ and
its time derivative are Schwartz class functions for all $t\in[0,T]$.
Then  a solution $v_2^\eps$ of the Cauchy problem
\begin{equation}\label{eq_lemmm}
\partial_t v-L_t^\eps v=\ell\Big(\frac x\eps,\frac t{\eps^2}\Big)U(x^\eps,t),\qquad v(x,0)=0,
\end{equation}
converges to zero in $L^\infty(0,T;L^2(\mathbb R^d))$ as $\eps\to0$.
\end{lemma}
We have all the necessary tools to study the asymptotic behaviour of $R^{\eps,(2)}$.  In the same way as in the symmetric
case one can show that the $L^2(\mathbb R^d)$ norms of the third and the forth  terms on the right-hand side  of \eqref{eq:drift_part_bis} tend to zero as $\eps\to0$.  Due to the a priori estimates given by
\cite[Proposition 4.1]{piat:zhiz:25}, 
these terms do not contribute to the limit of $R^{\eps,(2)}$.  Let us define the tensor $H^{\rm eff,ns}$ by:
\begin{equation}\label{def_H_nsym}
H^{\rm eff,ns}=\mathbb{E}\bigg(\int_{\mathbb T^d}\big( H(\xi,s) - \chi_1(\xi,s)\otimes \Theta^{\rm eff,ns}
- \chi_2(\xi,s)\otimes\beta\big)p(\xi,s)\,d\xi\bigg);
\end{equation}
Thus using the identity $\partial_t u^0=\Theta^{\rm eff,ns}\nabla\nabla u^0$, it follows from Lemmata \ref{l_hom_rhs}--\ref{l_x_osc_small}, with
\begin{align*}
U(x,t) & = \nabla \nabla \nabla u^0(x,t), \\
\theta (s) & =  \int_{\mathbb T^d}\big( H(\xi,s) - \chi_1(\xi,s)\otimes \Theta^{\rm eff,ns}
- \chi_2(\xi,s)\otimes\beta\big)p(\xi,s)\,d\xi - H^{\rm eff,ns}, \\
\ell \Big( \xi,s \Big) & = H\Big( \xi,s\Big)  - \chi_1 \left( \xi,s\right)  \otimes \Theta^{\rm eff,ns} -\chi_2\left( \xi,s \right)\otimes\beta  -  \theta (s)
\end{align*}
that $R^{\eps,(2)}$ converges in $L^\infty(0,T;L^2(\mathbb R^d))$ to a solution of the problem
\begin{equation}\label{lim_r2}
\partial_t v=\mathrm{div}\big(\Theta^{\rm eff,ns}\nabla v\big)=H^{\rm eff,ns}\nabla\nabla\nabla u^0.
\end{equation}

We turn to studying the limit  behaviour of $R^{\eps,(1)}$. As in the previous section we denote
$$
\kappa^\eps(t) =\dfrac{1}{\eps} \int_0^t \widetilde \Theta \left( \dfrac{s}{\eps^2} \right) ds,
$$
with $\widetilde \Theta(s)$ defined in \eqref{eq:effective_matrix_nonsym}.
Since $p(\xi,s)$ is measurable w.r.t. $\mathcal F_{\geq s}$, that is w.r.t. the future of $\Lambda$, the arguments
used in the proof of Lemma \ref{lem:CLT} should be rearranged.
\begin{lemma}\label{lem_clt_nonsym}
The process $\kappa^\eps$ converges in law to $A^{{\rm eff, ns}} W$ where $A^{{\rm eff, ns}}$ is a $d^2\times d^2$-matrix and $W$ a $d^2$-dimensional standard Brownian motion.
\end{lemma}
\begin{proof}
As before, we prove that \eqref{eq:mixing_property} holds. The main difference is that $p$ is itself stochastic and $p(\cdot,s)$ depends on $\mathcal F_{\geq s}$ ; thus we need to adapt the proof of Lemma \ref{lem:CLT}. Again let us fix $s > 0$ and define $\widehat \chi_1$ as the solution of: for $r \in [2s/3,s]$
$$
\partial_r \widehat \chi_1(\xi,r) - L_r \widehat \chi_1(\xi,r) = g(\xi,r)+\beta,\quad \widehat \chi_1(\xi,2s/3) = 0;
$$
here $g(\xi,r)$ is given by \eqref{def_g}.
Note that $ \widehat \chi_1(\cdot,s)$ is $\mathcal F_{\geq 2s/3}$-measurable. And using \cite[Lemma 3.5]{piat:zhiz:25}, we deduce that there exists $C > 0$ and $\gamma_0 > 0$ s.t.
$$\| \widehat \chi_1(\cdot,s) - \chi_1(\cdot,s) \|_{L^2(\mathbb T^d)} \leq C e^{-\gamma_0 s}.$$
We define $\widehat p$ as the solution of
$$-\partial_r \widehat p = L^*_r \widehat p$$
on $[0,s/3]$, with the terminal condition $\widehat p(\xi,s/3)=1$. For any $t \in [0,s/3]$, $\widehat p(\cdot,t)$ is measurable w.r.t. $\mathcal F_{\leq s/3}$. Moreover the difference $p-\widehat p$ satisfies:
$$
\|p(\cdot,0) - \widehat  p(\cdot,0) \|_{L^2(\mathbb T^d)} \leq C e^{-\gamma_0 s}.
$$

Now we break $\widetilde \Theta(s)$ down as follows: $\widetilde \Theta(s)=\widehat \Theta(s) +\mathring \Theta(s)$
with
\begin{align*}
\widehat \Theta(s)  & =  \int_{\mathbb T^d} \left(   \int_{\R^d} a\left( z\right) \Lambda\left( \xi,  \xi -z,  s \right)  \left( \dfrac{1}{2} (z \otimes z) - \widehat \chi_1 \left( \xi -z, s \right)z  \right) dz+\beta\otimes  \widehat \chi_1( \xi,s)
\right) p(\xi,s)  d\xi \\
&  - \Theta^{{\rm eff, ns}},  \\
\mathring \Theta(s) & =  \int_{\mathbb T^d} \bigg(   \int_{\R^d} a( z) \Lambda( \xi,  \xi -z,  s) \big( \widehat \chi_1 \big( \xi -z, s ) - \chi_1 ( \xi -z, s)   \big) z dz \\
&+\beta\otimes\big(\chi_1(\xi, s)-\widehat \chi_1(\xi, s)\big) \bigg) p(\xi,s)  d\xi .
\end{align*}
Notice that $\widehat \Theta(s)$ is measurable with respect to  the $\sigma$-algebra $\mathcal F_{\geq 2s/3}$, whereas $|\mathring \Theta(s) | \leq C e^{-\gamma_0 s}$. We similarly decompose $\widetilde \Theta(0)$ as
$$\widetilde \Theta(0)  = \widehat \Theta(0) +\mathring \Theta(0),$$
with
\begin{align*}
\widehat \Theta(0)  & =  \int_{\mathbb T^d} \left(   \int_{\R^d} a( z) \Lambda\left( \xi,  \xi -z,  0 \right)  \big( \dfrac{1}{2} (z \otimes z) - \chi_1( \xi -z, 0 )z  \big) dz+\beta\otimes\chi_1(\xi,0) \right) \widehat p(\xi,0)  d\xi \\
& - \Theta^{{\rm eff, ns}} \\
\mathring \Theta(0) & =  \int_{\mathbb T^d} \left(   \int_{\R^d} a\left( z\right) \Lambda\left( \xi,  \xi -z,  s \right)   \left( \dfrac{1}{2} (z \otimes z) - \chi_1 ( \xi -z, 0)z  \right) dz+\beta\otimes\chi_1(\xi,0)  \right) \\
&\quad  (p(\xi, 0) - \widehat p(\xi,0) )  d\xi .
\end{align*}
Here $\widehat \Theta(0)$ is measurable with respect to  the $\sigma$-algebra $\mathcal F_{\leq s/3}$, whereas $|\mathring \Theta(0) | \leq C e^{-\gamma_0 s}$.

Now we compute
\begin{align*}
\mathbb E ( \widetilde \Theta(s) \widetilde \Theta(0) )&  =\mathbb E ( \widehat \Theta(s) \widehat \Theta(0)) + \mathbb E ( \mathring \Theta(s) \widehat \Theta(0)) + \mathbb E ( \widehat \Theta(s) \mathring \Theta(0))  + \mathbb E ( \mathring \Theta(s) \mathring \Theta(0)) \\
&  =\mathbb E ( \widehat \Theta(s) \widehat \Theta(0)) - \mathbb E ( \widehat \Theta(s)) \mathbb E (\widehat \Theta(0))  +  \mathbb E ( \widehat \Theta(s)) \mathbb E (\widehat \Theta(0))  \\
& \quad + \mathbb E ( \mathring \Theta(s) \widehat \Theta(0)) + \mathbb E ( \widehat \Theta(s) \mathring \Theta(0))  + \mathbb E ( \mathring \Theta(s) \mathring \Theta(0)) .
\end{align*}
From Condition {\bf H3}, since $ \widehat \Theta(s)$ is measurable w.r.t. $\mathcal F_{\geq 2s/3}$ and $ \widehat \Theta(0)$ is measurable w.r.t. $\mathcal F_{\leq s/3}$, we have
$$|\mathbb E ( \widehat \Theta(s) \widehat \Theta(0)) - \mathbb E ( \widehat \Theta(s)) \mathbb E (\widehat \Theta(0)) | \leq \rho(s/3). $$
Moreover we know that for any $r$, $\mathbb E \widetilde \Theta(r) = 0$. Hence from the estimate on $\mathring \Theta(s)$
$$\mathbb E \widetilde \Theta(s) = 0 = \mathbb E \widehat \Theta(s)  + \mathbb E \mathring \Theta(s) \Rightarrow | \mathbb E \widehat \Theta(s)  | \leq C e^{-\gamma_0 s}. $$
And a similar upper bound holds for $| \mathbb E \widehat \Theta(0)  |$. We deduce that
$$
|\mathbb E ( \widetilde \Theta(s) \widetilde \Theta(0) )| \leq \rho(s/3) + C e^{-\gamma_0 s}.
$$
This yields \eqref{eq:mixing_property},  and the desired convergence can be justified in the same way as in the proof of Lemma \ref{lem:CLT}. Letting
$$
\mathcal{C}^{\rm ns}=\int_0^\infty \mathbf{E}\big(\widetilde\Theta(s)\otimes\widetilde\Theta(0)+\widetilde\Theta(0)\otimes(\widetilde\Theta(s)\big)ds,
$$
we conclude that $A^{{\rm eff, ns}}$ is a square root of $\mathcal{C}^{\rm ns}$.
\end{proof}

The statement of the last lemma implies in the same way as in the symmetric case that $R^{\eps,(1)}$ converges in law
in the space $L^2(0,T; L^2(\mathbb R^d))$ to a solution of  the problem
\begin{equation} \label{eq:limit_SPDE_without_drift_ns}
d v_t -{\rm div }( \Theta^{{\rm eff,ns}} \nabla v ) dt= A^{{\rm eff,ns}} dW_t \nabla\nabla u^0, \quad v(x,0) = 0.
\end{equation}
Combining the above statements we arrive at the following result:
\begin{theorem} \label{thm:main_result_non_symm}
Let conditions  {\rm \bf{H1}}--{\rm \bf{H4}} and {\bf H6} be fulfilled, and assume that  $\beta$ is given by \eqref{eq:def_beta}. Then the process $U^\eps$ defined by
$$
U^\eps(x,t) = \dfrac{u^\eps(x,t)-u^0(x^\eps,t)}{\eps}  -  \chi_1 \left( \dfrac{x}{\eps}, \dfrac{t}{\eps^2} \right) \nabla
u^0(x^\eps,t)
$$
with $x^\eps = x - \dfrac{\beta}{\eps}t$, converges in law in $L^2(\mathbb R^d\times(0,T))$ to the unique solution of the SPDE
\begin{equation}\label{eq:limit_SPDE_bis}
dv_t - {\rm div} ( \Theta^{{\rm eff ,ns}} \nabla v) dt = A^{{\rm eff, ns}} dW_t  \nabla \nabla u^0 + H^{{\rm eff, ns}} \nabla \nabla \nabla u^0 dt, \quad v(x,0)=0,
\end{equation}
where $u^0$ is the solution of
\begin{equation} \label{eq:homo_pb_bis}
\partial_t u^0 - {\rm div} ( \Theta^{{\rm eff, ns}} \nabla u^0) = 0,  \quad u^0(x,0)=\imath(x),
\end{equation}
and
\begin{itemize}
\item the constant positive definite matrix $\Theta^{{\rm eff, ns}}$ is defined in \eqref{eq:effective_matrix_nonsym};
\item the $d^2\times d^2$ matrix $A^{{\rm eff, ns}} $ is the symmetric positive semi-definite square root of
$\mathcal{C}^{\rm ns}$;
\item the matrix $H^{{\rm eff, ns}} $ is given by \eqref{def_H_nsym}.
\end{itemize}
\end{theorem}

\section{Appendix}

Here we provide the proof of Lemmata \ref{l_timeosc_small} and \ref{l_x_osc_small}.

\medskip\noindent
\begin{proof}[Proof of Lemma \ref{l_timeosc_small}]\\
Denote
$$
\kappa_1^\eps(t)=\int_0^t\theta\Big(\frac s{\eps^2}\Big)ds, \quad
\widetilde U^\eps(x,t)=\kappa^\eps_1(t) \Big[U(x^\eps,t)+\eps\chi_1\Big(\frac x\eps, \frac t{\eps^2}\Big)\nabla U(x^\eps,t)\Big].
$$
Substituting $(\widetilde U^\eps(x,t)-v_1^\eps)$ for $v$ in \eqref{eq_lem9} after straightforward rearrangements we obtain
\begin{align*}
& \partial_t(\widetilde U^\eps(x,t)-v_1^\eps)-L_t^\eps(\widetilde U^\eps(x,t)-v_1^\eps)=\eps\theta\Big(\frac t{\eps^2}\Big)
\chi_1\Big(\frac x\eps, \frac t{\eps^2}\Big)\nabla U(x^\eps,t) \\
& \quad
+\kappa_1^\eps(t) \bigg[\int_{\mathbb R^d}a(z)\Lambda\Big(\frac x\eps,\frac x\eps-z,\frac t{\eps^2}\Big)\bigg\{z\otimes z
\int_0^1\nabla\nabla U(x^\eps-\eps qz)(1-q)dq \\
& \quad  -\chi_1\Big(\frac x\eps-z, \frac t{\eps^2}\Big)\otimes z\int_0^1\nabla\nabla U(x^\eps-\eps qz)dq \bigg\}dz
\bigg]-\beta \kappa_1^\eps(t)  \chi_1\Big(\frac x\eps, \frac t{\eps^2}\Big)\nabla\nabla U(x^\eps,t)  \\
& \quad  +  \kappa_1^\eps(t) \Big[\partial_t U(x^\eps,t)+\eps\chi_1\Big(\frac x\eps,\frac t{\eps^2}\Big)\nabla\partial_t U(x^\eps,t)\Big].
\end{align*}
By the Birkhoff ergodic theorem and due to boundedness of $\theta(s)$, the process $\kappa_1^\eps(t)$ a.s. converges to
zero in $C[0,T]$. Therefore, the right-hand side in the last equation a.s. tends to zero in $L^\infty(0,T; L^2(\mathbb R^d))$.
Also we have $\widetilde U^\eps(x,0)-v_1^\eps(x,0)=\eps \chi_1(\frac x\eps,0)U(x,0)\to 0$ in $L^2(\mathbb R^d)$.
This implies that $\|\widetilde U^\eps(x,t)-v_1^\eps(x,t)\|_{L^\infty(0,T;L^2(\mathbb R^d))}\to0$ a.s. as $\eps\to0$,
which in turn yields the desired statement.
\end{proof}

\medskip\noindent
\begin{proof}[Proof of Lemma \ref{l_x_osc_small}]\\
Under the conditions of lemma the equation
$$
\partial_s\psi(\xi,s)-L_s\psi(\xi,s)=\ell(\xi,s)
$$
has a unique up to an additive constant stationary solution  $\psi\in L^\infty((-\infty,+\infty); L^2(\mathbb T^d))$.
Substituting the difference  $\big(\eps^2\psi(\frac x\eps,\frac t{\eps^2})U(x^\eps,t) -v_2^\eps(x,t)\big)$ for $v$
in \eqref{eq_lemmm}, in the same way as in the proof of the previous lemma we conclude  that
$$
\lim\limits_{\eps\to0}\|v_2^\eps(x,t)\|_{L^\infty(0,T;L^2(\mathbb R^d))}=
\lim\limits_{\eps\to0}\|v_2^\eps(x,t)-\eps^2\psi(\frac x\eps,\frac t{\eps^2})U(x^\eps,t)\|_{L^\infty(0,T;L^2(\mathbb R^d))}=0.
$$
\end{proof}

\bibliographystyle{abbrv}
\bibliography{recherchebib}

\end{document}